\documentclass[10pt,reqno,a4paper,oneside,11pt]{amsart}%
\usepackage{amsfonts}
\usepackage{amsfonts}
\usepackage{amsfonts}
\usepackage{amsfonts}
\usepackage{amsfonts}
\usepackage{amsfonts}
\usepackage{amsfonts}
\usepackage{amsfonts}
\usepackage{amsfonts}
\usepackage{amsfonts}
\usepackage{amsfonts}
\usepackage{amsfonts}
\usepackage{amsfonts}
\usepackage{amsfonts}
\usepackage{amsfonts}
\usepackage{amsfonts}
\usepackage{amsfonts}
\usepackage{amsfonts}
\usepackage{amsfonts}
\usepackage{amsfonts}
\usepackage{mathrsfs}
\usepackage{mathrsfs}
\usepackage{amsfonts}
\usepackage{amssymb}
\usepackage{amsmath}
\usepackage{amsthm}
\usepackage{graphicx}
\providecommand{\U}[1]{\protect\rule{.1in}{.1in}}
\newtheorem{theorem}{Theorem}[section]

\newtheorem{lemma}[theorem]{Lemma}

\newtheorem{remark}{Remark}[section]

\theoremstyle{definition}
\theoremstyle{remark}
\numberwithin{equation}{section}

\ifx\pdfoutput\relax\let\pdfoutput=\undefined\fi
\newcount\msipdfoutput
\ifx\pdfoutput\undefined\else
\ifcase\pdfoutput\else
\ifx\paperwidth\undefined\else
\ifdim\paperheight=0pt\relax\else\pdfpageheight\paperheight\fi
\ifdim\paperwidth=0pt\relax\else\pdfpagewidth\paperwidth\fi
\fi\fi\fi
\begin{document}
\pagestyle{myheadings}

\begin{center}
{\huge \textbf{A  simultaneous decomposition of seven  matrices over real quaternion algebra and its applications}}\footnote{This research was supported by
the grants from the National Natural
Science Foundation of China (11171205), the Natural Science Foundation of
Shanghai (11ZR1412500), the Key Project of Scientific Research Innovation Foundation
of Shanghai Municipal Education Commission (13ZZ080).
\par
{}* Corresponding author}

\bigskip

{\large \textbf{Zhuo-Heng He, Qing-Wen Wang$^{*}$}}

Department of Mathematics, Shanghai University, Shanghai 200444. P.R. China

E-mail: wqw@staff.shu.edu.cn (Q.W. Wang), hzh19871126@126.com (Z.H. He)

\end{center}

\begin{quotation}
\noindent\textbf{Abstract:} Let $\mathbb{H}$ be the real quaternion algebra
   and $\mathbb{H}^{n\times m}$ denote the set of all $n\times m$ matrices over $\mathbb{H}$. In this paper,  we  construct a  simultaneous decomposition of  seven general real quaternion matrices
with  compatible sizes: $A\in \mathbb{H}^{m\times n},
B\in \mathbb{H}^{m\times p_{1}},C\in \mathbb{H}^{m\times p_{2}},D\in \mathbb{H}^{m\times p_{3}},E\in \mathbb{H}^{q_{1}\times n},F\in \mathbb{H}^{q_{2}\times n},G\in \mathbb{H}^{q_{3}\times n}$. As applications of  the simultaneous
matrix decomposition,   we give solvability conditions, general solutions, as well as  the range of ranks    of the   general
solutions   to the following two real quaternion matrix equations $BXE+CYF+DZG=A$ and $BX+WE+CYF+DZG=A,$
where $A,B,C,D,E,F,$ and $G$ are given real quaternion matrices.
\end{quotation}

\section{\textbf{Introduction}}
Throughout this paper, let $\mathbb{R},\mathbb{C},$ and $\mathbb{H}^{m\times
n}$ stand, respectively, for the real number field, the complex number field, and the set of all $m\times n$ matrices over the real quaternion algebra
\[
\mathbb{H}=\big\{a_{0}+a_{1}i+a_{2}j+a_{3}k\big|~i^{2}=j^{2}=k^{2}%
=ijk=-1,a_{0},a_{1},a_{2},a_{3}\in\mathbb{R}\big\}.
\]
 The rank of a quaternion matrix $A$ is defined to be the maximum number
of columns of $A$ which are right linearly independent \cite{zhangfuzheng}. It is easy to see that for any nonsingular matrices $P$ and $Q$ of appropriate sizes, $A$ and $PAQ$ have the same rank \cite{zhangfuzheng}. Moreover, for $A\in \mathbb{H}^{m\times n},$ by \cite{TWH}, there exist invertible matrices $P$ and $Q$ such that
\begin{align*}
PAQ=\begin{pmatrix}I_{r}&0\\0&0\end{pmatrix}
\end{align*}where $r=r(A),$ $I_{r}$ is the $r\times r$ identity matrix. The rank of a block real quaternion matrix
\begin{align*}
\begin{pmatrix}A_{11}&A_{12}&\cdots&A_{1,n}\\A_{21}&A_{22}&\cdots&A_{2,n}\\ \vdots & \vdots &\ddots &\vdots \\
A_{m,1}&A_{m,2}&\cdots&A_{m,n}\end{pmatrix}
\end{align*}is denoted by $r_{a_{11}a_{12}\cdots a_{1,n}|a_{21}a_{22}\cdots a_{2,n}|\cdots |a_{m,1}a_{m,2}\cdots a_{m,n}}.$
If the block real quaternion matrix has zero, we use ``$0$" instead of zero in the subscript. For instance, the ranks of the following block real quaternion matrices
\begin{align*}
(B,~C,~D),~\begin{pmatrix}E\\F\\G\end{pmatrix},~\begin{pmatrix}D&0&A&B&0\\D&A&0&0&C\\0&E&0&0&0\\0&0&F&0&0\\0&F&0&0&0\end{pmatrix}
\end{align*}
are represented by $r_{bcd},~r_{e|f|g},~ r_{d0ab0|da00c|0e000|00f00|0f000}$, respectively. The set of all $n\times n$ invertible real quaternion matrices  is denoted by $GL_{n}(\mathbb{H})$.

Quaternions were introduced by Irish mathematician Sir William Rowan
Hamilton in 1843. It is well known that $\mathbb{H}$  is an associative and noncommutative division algebra.
 General
properties of quaternion and real quaternion matrices can be found in
\cite{zhangfuzheng}.  Nowadays
real quaternion matrices have always been at the heart of   computer science,
quantum physics, signal and color image processing, and so on (e.g.
\cite{N. LE Bihan}, \cite{J.C.K.Chu}, \cite{S. De Leo}, \cite{Took1}-\cite{Took4}).

In mathematics, engineering, signal, circuit and others, many problems can be
transformed into the decomposition of multiple matrices (e.g.
\cite{tree3},  \cite{siamre}, \cite{tree1}).
In 1981, Paige and Saunders \cite{ccp} introduced the generalized singular value decomposition of two matrices with the same row number. In 1991, Zha \cite{tree1} gave a restricted singular values of a general matrix triplet $\begin{pmatrix}\begin{smallmatrix}A&B\\C&~\end{smallmatrix}\end{pmatrix}$ over
$\mathbb{C}$. In 1991, Bart De Moor and G.H. Golub \cite{moor1} derived a generalization of the OSVD. Several applications were discussed in \cite{moor1}. Moreover, Bart De Moor and Paul Van Dooren \cite{moor3} presented the generalized singular value decompositions for $k$ general matrices $(A_{1},A_{2},\ldots,A_{k})$, where
$A_{1}\in \mathbb{C}^{n_{0}\times n_{1}},A_{2}\in \mathbb{C}^{n_{1}\times n_{2}},\ldots,A_{k}\in \mathbb{C}^{n_{k-1}\times n_{k}}.$ In 1994, C.C. Paige
and M.S. Wei \cite{ccpaige} introduced the history and generality of the CS Decomposition. In 2000, Delin Chu, Lieven De Lathauwer and Bart De Moor \cite{CHU5} proved that the restricted singular value decomposition of a general matrix triplet $\begin{pmatrix}\begin{smallmatrix}A&B\\C&~\end{smallmatrix}\end{pmatrix}$ can be computed using a CSD-based QR-type method. In 2011, Wang, van der Woude and Yu \cite{QWWangandyushaowen} derived the decomposition of three general matrices with the same row number over an arbitrary division ring. In 2012, Wang, Zhang and van der Woude \cite{zhangxia} gave a simultaneous decomposition concerning the general matrix quaternity $\begin{pmatrix}\begin{smallmatrix}A&B&C\\D&~&~\end{smallmatrix}\end{pmatrix}$ over an arbitrary division ring. Quite recently, He and Wang \cite{hezhuoheng} constructed a simultaneous decomposition of five general real quaternion matrices $\begin{pmatrix}\begin{smallmatrix}A&B&C\\D&~&~\\E&~&~\end{smallmatrix}\end{pmatrix}$. He and Wang \cite{hezhuoheng}
gave the rang of ranks of the real quaternion matrix expression $A-BXD-CYE$ by using the
simultaneous decomposition of five general real quaternion matrices $\begin{pmatrix}\begin{smallmatrix}A&B&C\\D&~&~\\E&~&~\end{smallmatrix}\end{pmatrix}$.

The remainder of the paper is organized as follows. In Section 2, we propose
the  simultaneous decomposition of the  matrix array
\begin{align}\label{array1}
\begin{pmatrix}
A&B&C&D\\
E&~&~&~\\
F&~&~&~\\
G&~&~&~\end{pmatrix},
\end{align}
where $A\in \mathbb{H}^{m\times n},
B\in \mathbb{H}^{m\times p_{1}},C\in \mathbb{H}^{m\times p_{2}},D\in \mathbb{H}^{m\times p_{3}},E\in \mathbb{H}^{q_{1}\times n},F\in \mathbb{H}^{q_{2}\times n}$  and $G\in \mathbb{H}^{q_{3}\times n}$ are general real quaternion matrices. In Section 3, we discuss several applications of the simultaneous decomposition. In Section 3.2, we derive   solvability conditions and the general solution to the real quaternion matrix equation
\begin{align}\label{system002}
BXE+CYF+DZG=A.
\end{align}
In Section 3.3, we give the range of ranks  of the general solution to the real quaternion matrix equation (\ref{system002}). In Section 3.4, we present
 solvability conditions and   the general solution to the real quaternion matrix equation
\begin{align}\label{system001}
BX+WE+CYF+DZG=A.
\end{align}
In Section 3.5, we derive the range of ranks  of the general solution to the real quaternion matrix equation (\ref{system001}).

\section{\textbf{A simultaneous decomposition of the general matrix array (\ref{array1}) over $\mathbb{H}$}}

We begin with the following lemma that is a  basic
tool for obtaining the main result.
\begin{lemma}\label{lemma00}\cite{QWWangandyushaowen}
Let $B\in \mathbb{H}^{m\times p_{1}},
C\in \mathbb{H}^{m\times p_{2}}$ and $D\in \mathbb{H}^{m\times p_{3}}$ be given.  Then there exist
 $\widetilde{P}\in GL_{m}(\mathbb{H}),\widetilde{T_{1}}\in GL_{p_{1}}(\mathbb{H}),
\widetilde{T_{2}}\in GL_{p_{2}}(\mathbb{H})$ and $\widetilde{T_{3}}\in GL_{p_{3}}(\mathbb{H})$ such that
\begin{align}
B=\widetilde{P} \widetilde{S_{B}}\widetilde{T_{1}}, \qquad C=\widetilde{P} \widetilde{S_{C}}\widetilde{T_{2}},\qquad D=\widetilde{P} \widetilde{S_{D}}\widetilde{T_{3}},
\end{align}
where
\begin{align*}
(\widetilde{S_{B}},\widetilde{S_{C}},\widetilde{S_{D}})=\begin{pmatrix}\begin{matrix} I &0&0&0&0&0\\
0 & I&0&0&0&0\\
0 & 0&I&0&0&0\\
0 & 0&0&I&0&0\\
0 & 0&0&0&I&0\\
0 & 0&0 &0&0&0\\
0 & 0&0 &0&0&0\\
0 & 0&0 &0&0&0\\
0 & 0&0 &0&0&0\\
0 & 0&0 &0&0&0
\end{matrix}~&
\begin{matrix} 0 &0&0&I&0&0\\
0 & 0&0&0&I&0\\
0 & 0&0 &0&0&0\\
0 & 0&0 &0&0&0\\
0 & 0&0 &0&0&0\\
I & 0&0&0&0&0\\
0 & I&0&0&0&0\\
0 & 0&I&0&0&0\\
0 & 0&0 &0&0&0\\
0 & 0&0 &0&0&0
\end{matrix}~&
\begin{matrix} 0 &0&0&0 &I&0\\
0 & 0&0 &0&0&0\\
0&0&0&I&0&0\\
0&I & 0&0&0&0\\
0 & 0&0 &0&0&0\\
0&I & 0&0&0&0\\
0 & 0&I&0&0&0\\
0 & 0&0 &0&0&0\\
I&0 & 0&0&0&0\\
0 & 0&0 &0&0&0
\end{matrix}
\end{pmatrix}
\begin{matrix} m_{1}\\
m_{2}\\
m_{3}\\
m_{4}\\
m_{5}\\
m_{4}\\
m_{6}\\
m_{7}\\
m_{8}\\
m-r_{bcd}
\end{matrix},
\end{align*}
where
\begin{align*}
m_{4}+m_{6}+m_{7}=r_{bc}-r_{b},~m_{1}+m_{2}=r_{b}+r_{c}-r_{bc},~m_{8}=r_{bcd}-r_{bc},
\end{align*}
\begin{align*}
m_{4}+m_{6}=r_{bc}+r_{bd}-r_{bcd}-r_{b},~m_{1}+m_{3}=r_{b}+r_{d}-r_{bd},~m_{3}+m_{4}=r_{bc}+r_{cd}-r_{bcd}-r_{c}.
\end{align*}

\end{lemma}

Now we give the main theorem of this section.

\begin{theorem}\label{theorem01}
Let $A\in \mathbb{H}^{m\times n},
B\in \mathbb{H}^{m\times p_{1}},C\in \mathbb{H}^{m\times p_{2}},D\in \mathbb{H}^{m\times p_{3}},E\in \mathbb{H}^{q_{1}\times n},F\in \mathbb{H}^{q_{2}\times n}$  and $G\in \mathbb{H}^{q_{3}\times n}$ be given. Then there exist
 $P\in GL_{m}(\mathbb{H}),~Q\in GL_{n}(\mathbb{H}),~T_{1}\in GL_{p_{1}}(\mathbb{H}),~
T_{2}\in GL_{p_{2}}(\mathbb{H}),~T_{3}\in GL_{p_{3}}(\mathbb{H}),~V_{1}\in GL_{q_{1}}(\mathbb{H}),~V_{2}\in GL_{q_{2}}(\mathbb{H}),$ and $V_{3}\in GL_{q_{3}}(\mathbb{H})$ such that%
\begin{align}\label{equ021}
A=P S_{A}Q,B=P S_{B}T_{1}, C=P S_{C}T_{2},D=P S_{D}T_{3}, E=V_{1} S_{E}Q, F=V_{2} S_{F}Q,G=V_{3} S_{G}Q,
\end{align}
where%
\begin{align}\label{equ0022}
S_{A}=\begin{pmatrix}
A_{11} & \cdots&   A_{19}&   A_{1,10}&0\\
\vdots& \ddots&   \vdots&\vdots&\vdots\\
A_{91}& \cdots&A_{99}&A_{9,10}& 0\\
A_{10,1}& \cdots&A_{10,9}&0&0\\0&\cdots&0 &0&I_{t}
\end{pmatrix},
\end{align}
\begin{align}\label{equ0023}
(S_{B},S_{C},S_{D})=\begin{pmatrix}\begin{matrix} I &0&0&0&0&0\\
0 & I&0&0&0&0\\
0 & 0&I&0&0&0\\
0 & 0&0&I&0&0\\
0 & 0&0&0&I&0\\
0 & 0&0 &0&0&0\\
0 & 0&0 &0&0&0\\
0 & 0&0 &0&0&0\\
0 & 0&0 &0&0&0\\
0 & 0&0 &0&0&0\\
0 & 0&0 &0&0&0
\end{matrix}~&
\begin{matrix} 0 &0&0&I&0&0\\
0 & 0&0&0&I&0\\
0 & 0&0 &0&0&0\\
0 & 0&0 &0&0&0\\
0 & 0&0 &0&0&0\\
I & 0&0&0&0&0\\
0 & I&0&0&0&0\\
0 & 0&I&0&0&0\\
0 & 0&0 &0&0&0\\
0 & 0&0 &0&0&0\\
0 & 0&0 &0&0&0
\end{matrix}~&
\begin{matrix} 0 &0&0&0 &I&0\\
0 & 0&0 &0&0&0\\
0&0&0&I&0&0\\
0&I & 0&0&0&0\\
0 & 0&0 &0&0&0\\
0&I & 0&0&0&0\\
0 & 0&I&0&0&0\\
0 & 0&0 &0&0&0\\
I&0 & 0&0&0&0\\
0 & 0&0 &0&0&0\\
0 & 0&0 &0&0&0
\end{matrix}
\end{pmatrix}
\begin{matrix} m_{1}\\
m_{2}\\
m_{3}\\
m_{4}\\
m_{5}\\
m_{4}\\
m_{6}\\
m_{7}\\
m_{8}\\
m-r_{bcd}-t\\
t
\end{matrix},
\end{align}
\begin{align}\label{equ0024}
(S_{E}^{*},S_{F}^{*},S_{G}^{*})=\begin{pmatrix}\begin{matrix} I &0&0&0&0&0\\
0 & I&0&0&0&0\\
0 & 0&I&0&0&0\\
0 & 0&0&I&0&0\\
0 & 0&0&0&I&0\\
0 & 0&0 &0&0&0\\
0 & 0&0 &0&0&0\\
0 & 0&0 &0&0&0\\
0 & 0&0 &0&0&0\\
0 & 0&0 &0&0&0\\
0 & 0&0 &0&0&0
\end{matrix}~&
\begin{matrix} 0 &0&0&I&0&0\\
0 & 0&0&0&I&0\\
0 & 0&0 &0&0&0\\
0 & 0&0 &0&0&0\\
0 & 0&0 &0&0&0\\
I & 0&0&0&0&0\\
0 & I&0&0&0&0\\
0 & 0&I&0&0&0\\
0 & 0&0 &0&0&0\\
0 & 0&0 &0&0&0\\
0 & 0&0 &0&0&0
\end{matrix}~&
\begin{matrix} 0 &0&0&0 &I&0\\
0 & 0&0 &0&0&0\\
0&0&0&I&0&0\\
0&I & 0&0&0&0\\
0 & 0&0 &0&0&0\\
0&I & 0&0&0&0\\
0 & 0&I&0&0&0\\
0 & 0&0 &0&0&0\\
I&0 & 0&0&0&0\\
0 & 0&0 &0&0&0\\
0 & 0&0 &0&0&0
\end{matrix}
\end{pmatrix}
\begin{matrix}
n_{1}\\
n_{2}\\
n_{3}\\
n_{4}\\
n_{5}\\
n_{4}\\
n_{6}\\
n_{7}\\
n_{8}\\
n-r_{e|f|g}-t\\t
\end{matrix},
\end{align}
where
\begin{align*}
t=r_{abcd|e000|f000|g000}-r_{bcd}-r_{e|f|g},
\end{align*}
\begin{align}\label{equh025}
m_{1}=r_{b}+r_{c}+r_{d}-r_{db0|d0c},
~
m_{2}=r_{db0|d0c}-r_{bc}-r_{d},
\end{align}
\begin{align}
m_{3}=r_{db0|d0c}-r_{bd}-r_{c},
~
m_{4}=r_{bc}+r_{cd}+r_{bd}-r_{bcd}-r_{db0|d0c},
\end{align}
\begin{align}
m_{5}=r_{bcd}-r_{cd},~
m_{6}=r_{db0|d0c}-r_{cd}-r_{b},~
~
m_{7}=r_{bcd}-r_{bd},~
m_{8}=r_{bcd}-r_{bc},
\end{align}
\begin{align}
n_{1}=r_{e}+r_{f}+r_{g}-r_{gg|e0|0f},
~
n_{2}=r_{gg|e0|0f}-r_{e|f}-r_{g},
\end{align}
\begin{align}
n_{3}=r_{gg|e0|0f}-r_{e|f}-r_{f},~
n_{4}=r_{e|f}+r_{f|g}+r_{e|g}-r_{e|f|g}-r_{gg|e0|0f},
\end{align}
\begin{align}\label{equh0210}
n_{5}=r_{e|f|g}-r_{f|g},~
n_{6}=r_{gg|e0|0f}-r_{f|g}-r_{e},~
~
n_{7}=r_{e|f|g}-r_{e|g},~
n_{8}=r_{e|f|g}-r_{e|f}.
\end{align}
The block columns of $S_{A}$  are $(n_{1},n_{2},n_{3},n_{4},n_{5},n_{4},n_{6}.n_{7},n_{8},n-r_{e|f|g}-t,t).$
The block rows of $S_{A}$  are $(m_{1},m_{2},m_{3},m_{4},m_{5},m_{4},m_{6}.m_{7},m_{8},m-r_{bcd}-t,t).$

\end{theorem}

\begin{proof}The proof is constructive. We establish the result through the following steps. First, we give equivalence
canonical forms of the general matrix arrays $(B,~C,~D)$ and $\begin{pmatrix}\begin{smallmatrix}E\\F\\G\end{smallmatrix}\end{pmatrix}$. Second,  we provide the  simultaneous decomposition of the   matrix array (\ref{array1}).

Step 1. For the matrix arrays $(B,~C,~D)$ and $\begin{pmatrix}\begin{smallmatrix}E\\F\\G\end{smallmatrix}\end{pmatrix}$, we can find eight   matrices
$P_{1}\in GL_{m}(\mathbb{H}),$ $Q_{1}\in GL_{n}(\mathbb{H}),$ $W_{B}\in GL_{p_{1}}(\mathbb{H}),$ $
W_{C}\in GL_{p_{2}}(\mathbb{H}),$ $W_{D}\in GL_{p_{3}}(\mathbb{H}),$ $W_{E}\in GL_{q_{1}}(\mathbb{H}),$ $W_{F}\in GL_{q_{2}}(\mathbb{H}),$ $W_{G}\in GL_{q_{3}}(\mathbb{H})$ such that
\begin{align*}
P_{1}\begin{pmatrix}B&C&D\end{pmatrix}\begin{pmatrix}W_{B}&0&0\\0&W_{C}&0\\0&0&W_{D}\end{pmatrix}
=\end{align*}
\begin{align*}
\begin{pmatrix}\begin{matrix} I &0&0&0&0&0\\
0 & I&0&0&0&0\\
0 & 0&I&0&0&0\\
0 & 0&0&I&0&0\\
0 & 0&0&0&I&0\\
0 & 0&0 &0&0&0\\
0 & 0&0 &0&0&0\\
0 & 0&0 &0&0&0\\
0 & 0&0 &0&0&0\\
0 & 0&0 &0&0&0
\end{matrix}~&
\begin{matrix} 0 &0&0&I&0&0\\
0 & 0&0&0&I&0\\
0 & 0&0 &0&0&0\\
0 & 0&0 &0&0&0\\
0 & 0&0 &0&0&0\\
I & 0&0&0&0&0\\
0 & I&0&0&0&0\\
0 & 0&I&0&0&0\\
0 & 0&0 &0&0&0\\
0 & 0&0 &0&0&0
\end{matrix}~&
\begin{matrix} 0 &0&0&0 &I&0\\
0 & 0&0 &0&0&0\\
0&0&0&I&0&0\\
0&I & 0&0&0&0\\
0 & 0&0 &0&0&0\\
0&I & 0&0&0&0\\
0 & 0&I&0&0&0\\
0 & 0&0 &0&0&0\\
I&0 & 0&0&0&0\\
0 & 0&0 &0&0&0
\end{matrix}
\end{pmatrix}
\begin{matrix} m_{1}\\
m_{2}\\
m_{3}\\
m_{4}\\
m_{5}\\
m_{4}\\
m_{6}\\
m_{7}\\
m_{8}\\
m-r_{bcd}
\end{matrix},
\end{align*}
\begin{align*}
\left[\begin{pmatrix}W_{E}&0&0\\0&W_{F}&0\\0&0&W_{G}\end{pmatrix}\begin{pmatrix}E\\F\\G\end{pmatrix}Q_{1}\right]^{*}
=\end{align*}
\begin{align*}
\begin{pmatrix}\begin{matrix} I &0&0&0&0&0\\
0 & I&0&0&0&0\\
0 & 0&I&0&0&0\\
0 & 0&0&I&0&0\\
0 & 0&0&0&I&0\\
0 & 0&0 &0&0&0\\
0 & 0&0 &0&0&0\\
0 & 0&0 &0&0&0\\
0 & 0&0 &0&0&0\\
0 & 0&0 &0&0&0
\end{matrix}~&
\begin{matrix} 0 &0&0&I&0&0\\
0 & 0&0&0&I&0\\
0 & 0&0 &0&0&0\\
0 & 0&0 &0&0&0\\
0 & 0&0 &0&0&0\\
I & 0&0&0&0&0\\
0 & I&0&0&0&0\\
0 & 0&I&0&0&0\\
0 & 0&0 &0&0&0\\
0 & 0&0 &0&0&0
\end{matrix}~&
\begin{matrix} 0 &0&0&0 &I&0\\
0 & 0&0 &0&0&0\\
0&0&0&I&0&0\\
0&I & 0&0&0&0\\
0 & 0&0 &0&0&0\\
0&I & 0&0&0&0\\
0 & 0&I&0&0&0\\
0 & 0&0 &0&0&0\\
I&0 & 0&0&0&0\\
0 & 0&0 &0&0&0
\end{matrix}
\end{pmatrix}
\begin{matrix}
n_{1}\\
n_{2}\\
n_{3}\\
n_{4}\\
n_{5}\\
n_{4}\\
n_{6}\\
n_{7}\\
n_{8}\\
n-r_{e|f|g}
\end{matrix}.
\end{align*}
Let
\begin{align*}
P_{1}AQ_{1}\triangleq
\begin{pmatrix}
A_{11}^{(1)} & \cdots&   A_{1,10}^{(1)}\\
\vdots& \ddots&   \vdots\\
A_{10,1}^{(1)} & \cdots&A_{10,10}^{(1)}
\end{pmatrix},
\end{align*}
where the symbol $\triangleq$ means ``equals by definition''.  For the matrix $A_{10,10}^{(1)}$, we can find
$P_{2}\in GL_{m-r_{bcd}}(\mathbb{H}),$ $Q_{2}\in GL_{n-r_{e|f|g}}(\mathbb{H})$ such that
\begin{align*}
P_{2}A_{10,10}^{(1)}Q_{2}=\begin{pmatrix}0&0\\0&I_{t}\end{pmatrix},t\triangleq r(A_{10,10}^{(1)}).
\end{align*}
Then we have
\begin{align*}
\begin{pmatrix}I_{r_{bcd}}&0\\0&P_{2}\end{pmatrix}\begin{pmatrix}
A_{11}^{(1)} & \cdots&   A_{1,10}^{(1)}\\
\vdots& \ddots&   \vdots\\
A_{10,1}^{(1)} & \cdots&A_{10,10}^{(1)}
\end{pmatrix}\begin{pmatrix}I_{r_{e|f|g}}&0\\0&Q_{2}\end{pmatrix}\triangleq
\begin{pmatrix}
A_{11}^{(2)} & \cdots&   A_{19}^{(2)}&   A_{1,10}^{(2)}& A_{1,11}^{(2)}\\
\vdots& \ddots&   \vdots&\vdots&\vdots\\
A_{91}^{(2)} & \cdots&A_{99}^{(2)}&A_{9,10}^{(2)}& A_{9,11}^{(2)}\\
A_{10,1}^{(2)} & \cdots&A_{10,9}^{(2)}&0&0\\A_{11,1}^{(2)} &\cdots&A_{11,9}^{(2)} &0&I_{t}
\end{pmatrix},
\end{align*}
\begin{align*}
\begin{pmatrix} I_{r_{bcd}}&0\\0&P_{2} \end{pmatrix}P_{1}\begin{pmatrix}B&C&D\end{pmatrix}\begin{pmatrix}
 W_{B}&0&0\\0&W_{C}&0\\0&0&W_{D} \end{pmatrix}
=
\end{align*}
\begin{align*}
\begin{pmatrix}\begin{matrix} I &0&0&0&0&0\\
0 & I&0&0&0&0\\
0 & 0&I&0&0&0\\
0 & 0&0&I&0&0\\
0 & 0&0&0&I&0\\
0 & 0&0 &0&0&0\\
0 & 0&0 &0&0&0\\
0 & 0&0 &0&0&0\\
0 & 0&0 &0&0&0\\
0 & 0&0 &0&0&0\\
0 & 0&0 &0&0&0
\end{matrix}~&
\begin{matrix} 0 &0&0&I&0&0\\
0 & 0&0&0&I&0\\
0 & 0&0 &0&0&0\\
0 & 0&0 &0&0&0\\
0 & 0&0 &0&0&0\\
I & 0&0&0&0&0\\
0 & I&0&0&0&0\\
0 & 0&I&0&0&0\\
0 & 0&0 &0&0&0\\
0 & 0&0 &0&0&0\\
0 & 0&0 &0&0&0
\end{matrix}~&
\begin{matrix} 0 &0&0&0 &I&0\\
0 & 0&0 &0&0&0\\
0&0&0&I&0&0\\
0&I & 0&0&0&0\\
0 & 0&0 &0&0&0\\
0&I & 0&0&0&0\\
0 & 0&I&0&0&0\\
0 & 0&0 &0&0&0\\
I&0 & 0&0&0&0\\
0 & 0&0 &0&0&0\\
0 & 0&0 &0&0&0
\end{matrix}
\end{pmatrix}
\begin{matrix} m_{1}\\
m_{2}\\
m_{3}\\
m_{4}\\
m_{5}\\
m_{4}\\
m_{6}\\
m_{7}\\
m_{8}\\
m-r_{bcd}-t\\
t
\end{matrix},
\end{align*}
\begin{align*}
\left[\begin{pmatrix} W_{E}&0&0\\0&W_{F}&0\\0&0&W_{G} \end{pmatrix}\begin{pmatrix}
 E\\F\\G \end{pmatrix}Q_{1}
\begin{pmatrix} I_{r_{e|f|g}}&0\\0&Q_{2} \end{pmatrix}\right]^{*}
=
\end{align*}
\begin{align*}
\begin{pmatrix}\begin{matrix} I &0&0&0&0&0\\
0 & I&0&0&0&0\\
0 & 0&I&0&0&0\\
0 & 0&0&I&0&0\\
0 & 0&0&0&I&0\\
0 & 0&0 &0&0&0\\
0 & 0&0 &0&0&0\\
0 & 0&0 &0&0&0\\
0 & 0&0 &0&0&0\\
0 & 0&0 &0&0&0\\
0 & 0&0 &0&0&0
\end{matrix}~&
\begin{matrix} 0 &0&0&I&0&0\\
0 & 0&0&0&I&0\\
0 & 0&0 &0&0&0\\
0 & 0&0 &0&0&0\\
0 & 0&0 &0&0&0\\
I & 0&0&0&0&0\\
0 & I&0&0&0&0\\
0 & 0&I&0&0&0\\
0 & 0&0 &0&0&0\\
0 & 0&0 &0&0&0\\
0 & 0&0 &0&0&0
\end{matrix}~&
\begin{matrix} 0 &0&0&0 &I&0\\
0 & 0&0 &0&0&0\\
0&0&0&I&0&0\\
0&I & 0&0&0&0\\
0 & 0&0 &0&0&0\\
0&I & 0&0&0&0\\
0 & 0&I&0&0&0\\
0 & 0&0 &0&0&0\\
I&0 & 0&0&0&0\\
0 & 0&0 &0&0&0\\
0 & 0&0 &0&0&0
\end{matrix}
\end{pmatrix}
\begin{matrix}
n_{1}\\
n_{2}\\
n_{3}\\
n_{4}\\
n_{5}\\
n_{4}\\
n_{6}\\
n_{7}\\
n_{8}\\
n-r_{e|f|g}-t\\t
\end{matrix}.
\end{align*}
Let
\begin{align*}
P_{3}=\begin{pmatrix}I_{r_{bcd}}&\begin{pmatrix}0&-A_{1,11}^{(2)}\\
\vdots&\vdots\\0&-A_{9,11}^{(2)}\end{pmatrix}\\0&I_{m-r_{bcd}}\end{pmatrix},~
Q_{3}=\begin{pmatrix}I_{r_{e|f|g}}&0\\
\begin{pmatrix}0&\cdots&0\\
-A_{11,1}^{(2)}&\cdots&-A_{11,9}^{(2)}\end{pmatrix}&I_{n-r_{e|f|g}}\end{pmatrix}.
\end{align*}
Then we obtain
\begin{align*}
P_{3}\begin{pmatrix}
A_{11}^{(2)} & \cdots&   A_{19}^{(2)}&   A_{1,10}^{(2)}& A_{1,11}^{(2)}\\
\vdots& \ddots&   \vdots&\vdots&\vdots\\
A_{91}^{(2)} & \cdots&A_{99}^{(2)}&A_{9,10}^{(2)}& A_{9,11}^{(2)}\\
A_{10,1}^{(2)} & \cdots&A_{10,9}^{(2)}&0&0\\A_{11,1}^{(2)} &\cdots&A_{11,9}^{(2)} &0&I_{t}
\end{pmatrix}Q_{3}\triangleq
\begin{pmatrix}
A_{11} & \cdots&   A_{19}&   A_{1,10}&0\\
\vdots& \ddots&   \vdots&\vdots&\vdots\\
A_{91}& \cdots&A_{99}&A_{9,10}& 0\\
A_{10,1}& \cdots&A_{10,9}&0&0\\0 &\cdots&0 &0&I_{t}
\end{pmatrix}.
\end{align*}
Let
\begin{align*}
P\triangleq P_{3}\begin{pmatrix}I_{r_{bcd}}&0\\0&P_{2}\end{pmatrix}P_{1},~
Q\triangleq Q_{1}\begin{pmatrix}I_{r_{e|f|g}}&0\\0&Q_{2}\end{pmatrix}Q_{3},
\end{align*}
\begin{align*}
T_{1}=W_{C},~T_{2}=W_{D},~T_{3}=W_{E},~V_{1}=W_{E},~V_{2}=W_{F},~V_{3}=W_{G}.
\end{align*}
Hence,  the matrices  $P\in GL_{m}(\mathbb{H}),~Q\in GL_{n}(\mathbb{H}),~T_{1}\in GL_{p_{1}}(\mathbb{H}),~
T_{2}\in GL_{p_{2}}(\mathbb{H}),~T_{3}\in GL_{p_{3}}(\mathbb{H}),~V_{1}\in GL_{q_{1}}(\mathbb{H}),~V_{2}\in GL_{q_{2}}(\mathbb{H}),~V_{3}\in GL_{q_{3}}(\mathbb{H})$ satisfy the equation (\ref{equ021}). It follows from $S_{A},S_{B},S_{C},S_{D},S_{E},S_{F},$ and $S_{G}$ in
(\ref{equ0022})-(\ref{equ0024}) that
\begin{align*}
\begin{pmatrix}
1&1&1&1&1&0&0&0\\
1&1&0&1&0&1&1&0\\
1&0&1&1&0&1&0&1\\
1&1&1&2&1&1&1&0\\
1&1&1&2&1&1&0&1\\
1&1&1&2&0&1&1&1\\
1&1&1&2&1&1&1&1\\
1&0&1&1&0&1&0&1\end{pmatrix}\begin{pmatrix}m_{1}\\
m_{2}\\
m_{3}\\
m_{4}\\
m_{5}\\
m_{6}\\
m_{7}\\
m_{8}\end{pmatrix}=\begin{pmatrix}r_{b}\\
r_{c}\\
r_{d}\\
r_{bc}\\
r_{bd}\\
r_{cd}\\
r_{bcd}\\
r_{db0|d0c}-r_{b}-r_{c}\end{pmatrix},
\end{align*}
\begin{align*}
\begin{pmatrix}
1&1&1&1&1&0&0&0\\
1&1&0&1&0&1&1&0\\
1&0&1&1&0&1&0&1\\
1&1&1&2&1&1&1&0\\
1&1&1&2&1&1&0&1\\
1&1&1&2&0&1&1&1\\
1&1&1&2&1&1&1&1\\
1&0&1&1&0&1&0&1\end{pmatrix}\begin{pmatrix}n_{1}\\
n_{2}\\
n_{3}\\
n_{4}\\
n_{5}\\
n_{6}\\
n_{7}\\
n_{8}\end{pmatrix}=\begin{pmatrix}r_{e}\\
r_{f}\\
r_{g}\\
r_{e|f}\\
r_{e|g}\\
r_{f|g}\\
r_{e|f|g}\\
r_{gg|e0|0f}-r_{e}-r_{f}\end{pmatrix}.
\end{align*}
Solving for $m_{i},n_{i},(i=1,\ldots,8)$ gives (\ref{equh025})-(\ref{equh0210}).
\end{proof}

\begin{remark}
Wang et. al. \cite{QWWangandyushaowen} did not give the values of $m_{i},~(i=1,2,\ldots,8)$ in Lemma \ref{lemma00}. As a special case of
Theorem \ref{theorem01}, we can derive all the
dimensions of identity matrices in the   equivalence canonical form of triple real quaternion matrices $(B,C,D)$, i.e.,  the values of $m_{i},~(i=1,2,\ldots,8)$ in Lemma \ref{lemma00}:
\begin{align*}
m_{1}=r_{b}+r_{c}+r_{d}-r_{db0|d0c},
~
m_{2}=r_{db0|d0c}-r_{bc}-r_{d},
\end{align*}
\begin{align*}
m_{3}=r_{db0|d0c}-r_{bd}-r_{c},
~
m_{4}=r_{bc}+r_{cd}+r_{bd}-r_{bcd}-r_{db0|d0c},
\end{align*}
\begin{align*}
m_{5}=r_{bcd}-r_{cd},~
m_{6}=r_{db0|d0c}-r_{cd}-r_{b},~
~
m_{7}=r_{bcd}-r_{bd},~
m_{8}=r_{bcd}-r_{bc}.
\end{align*}On the other hand, the values of $m_{i},n_{i},(i=1,\ldots,8)$ play an important role in investigating
the range of ranks    of the   general solutions to  (\ref{system002}) and (\ref{system001}).
\end{remark}

\section{\textbf{Some applications of the simultaneous decomposition of  (\ref{array1})}}

The simultaneous decomposition of (\ref{array1}) is useful in solving  the following questions:

\begin{itemize}
  \item  \S 3.2.   Give some solvability conditions and an expression of the general solution to the real quaternion matrix equation (\ref{system002}).

  \item \S 3.3.  Give the range of ranks  of the general solution   in the real quaternion matrix equation (\ref{system002}).

  \item \S 3.4. Give some solvability conditions and an expression of the general solution to the real quaternion matrix equation (\ref{system001}).

  \item \S 3.5. Give the range of ranks  of the general solution   in the real quaternion matrix equation (\ref{system001}).

\end{itemize}

\subsection{\textbf{Preliminaries}}

In this section, we give some lemmas which are used in the further development of this paper. The following Lemmas are due to
\cite{CHU3}, \cite{cohen1} and  \cite{Woerdeman1}-\cite{Woerdeman3} which can be generalized to $\mathbb{H}.$
\begin{lemma}\label{lemma04}(\cite{CHU3}-\cite{hezhuoheng}, \cite{Woerdeman1}-\cite{Woerdeman3})
Let
\begin{align}
H(X,Y)=\begin{pmatrix}A_{1}&B_{1}&C_{1}\\D_{1}&X&E_{1}\\
F_{1}&G_{1}&Y\end{pmatrix},
\end{align}
 where $A_{1}\in \mathbb{H}^{\tilde{n}\times n},B_{1}\in \mathbb{H}^{\tilde{n}\times m},C_{1}\in \mathbb{H}^{\tilde{n}\times p},
 D_{1}\in \mathbb{H}^{\tilde{m}\times n},E_{1}\in \mathbb{H}^{\tilde{m}\times p},F_{1}\in \mathbb{H}^{\tilde{p}\times n}$ and $G_{1}\in \mathbb{H}^{\tilde{p}\times m}$ are given, and
$X\in \mathbb{H}^{\tilde{m}\times m}$ and $Y\in \mathbb{H}^{\tilde{p}\times p}$ are variable matrices. Then,
\begin{align*}
\mathop {\max }\limits_{ X,Y } r\left[ {H\left( {X,Y} \right)} \right]=\min \left\{\tilde{m}+\tilde{p}+r_{a_{1}b_{1}c_{1}},
\tilde{m}+p+r_{a_{1}b_{1}|f_{1}g_{1}},m+\tilde{p}+r_{a_{1}c_{1}|d_{1}e_{1}},m+p+r_{a_{1}|d_{1}|f_{1}}\right\},
\end{align*}
\begin{align*}
\mathop {\min }\limits_{ X,Y } r\left[ {H\left( {X,Y} \right)} \right]=r_{a_{1}b_{1}c_{1}}+r_{a_{1}|d_{1}|f_{1}}+
\max \left\{r_{a_{1}c_{1}|d_{1}e_{1}}-r_{a_{1}c_{1}}-r_{a_{1}|d_{1}},r_{a_{1}b_{1}|f_{1}g_{1}}-r_{a_{1}b_{1}}-r_{a_{1}|f_{1}} \right\}.
\end{align*}
\end{lemma}

\begin{lemma}\label{lemma01}(\cite{CHU3}-\cite{hezhuoheng}, \cite{Woerdeman1}-\cite{Woerdeman3})
Let
\begin{align}\label{equ031}
M(X,Y)=\begin{pmatrix}A_{1}&X\\Y&B_{1}\end{pmatrix}
\end{align}
 where $A_{1}\in \mathbb{H}^{m\times n}$ and $B_{1}\in  \mathbb{H}^{p\times q}$ are given, and
$X\in \mathbb{H}^{m\times q}$ and $Y\in \mathbb{H}^{p\times n}$ are variable matrices. Then,
\begin{align*}
\mathop {\max }\limits_{ X,Y } r\left[ {M\left( {X,Y} \right)} \right]=\min \left\{m+p,n+q,r_{a_{1}}+p+q,r_{b_{1}}+m+n \right\},
\end{align*}
\begin{align*}
\mathop {\min }\limits_{ X,Y } r\left[ {M\left( {X,Y} \right)} \right]=\max \left\{r_{a_{1}},~r_{b_{1}} \right\}.
\end{align*}
\end{lemma}

\begin{lemma}\label{lemma03}(\cite{cohen1}, \cite{hezhuoheng}, \cite{Woerdeman1}-\cite{Woerdeman3})
Let
\begin{align*}
M_{2}=\begin{pmatrix}Y&D_{1}\\B_{1}&A_{1}\end{pmatrix},
\end{align*}where $A_{1},B_{1}$ and $D_{1}$ are given, and $Y\in \mathbb{H}^{n\times m}$ is
a variable matrix. Then,
\begin{align*}
\mathop {\max }\limits_{Y\in \mathbb{H}^{n\times m} } r\left( {M_{2}} \right)=\mathop {\min }\left\{ n+r_{a_{1}b_{1}},m+r_{a_{1}|d_{1}} \right\},~
\mathop {\min }\limits_{Y\in \mathbb{H}^{n\times m} } r\left( {M_{2}} \right)=r_{a_{1}b_{1}}+r_{a_{1}|d_{1}}-r_{a_{1}}.
\end{align*}
\end{lemma}

\subsection{\textbf{Some solvability conditions and the   general
solution to  (\ref{system002})}}
In this section, the simultaneous decomposition of  (\ref{array1}) will be used to   solve the real quaternion matrix
equation (\ref{system002}).

\begin{theorem}\label{theorem04}
Let $A\in \mathbb{H}^{m\times n},
B\in \mathbb{H}^{m\times p_{1}},C\in \mathbb{H}^{m\times p_{2}},D\in \mathbb{H}^{m\times p_{3}},E\in \mathbb{H}^{q_{1}\times n},F\in \mathbb{H}^{q_{2}\times n}$  and $G\in \mathbb{H}^{q_{3}\times n}$ be given. Then the  equation (\ref{system002}) is consistent if and only if
\begin{align*}
A_{94}=A_{96},~A_{49}=A_{69},~A_{64}=A_{46},
\end{align*}
\begin{align*}
r_{abcd|e000|f000|g000}=r_{bcd}+r_{e|f|g},\quad \left(A_{1,10}^{*},~  \cdots, ~A_{9,10}^{*}\right)=0,\quad
\left(A_{10,1},~  \cdots, ~A_{10,9}\right)=0,
\end{align*}
\begin{align*}
A_{29}=0,~A_{92}=0,~A_{38}=0,~A_{83}=0,~A_{48}=0,~A_{84}=0,~A_{56}=0,~A_{65}=0,
\end{align*}
\begin{align}\label{equ00401}
A_{57}=0,~A_{75}=0,~A_{58}=0,~A_{85}=0,~A_{59}=0,~A_{95}=0,~A_{89}=0,~A_{98}=0.
\end{align}

In this case, the general solution to (\ref{system002}) can be expressed as
\begin{align*}
X=T_{1}^{-1}\widehat{X}V_{1}^{-1},\quad Y=T_{2}^{-1}\widehat{Y}V_{2}^{-1},\quad Z=T_{3}^{-1}\widehat{Z}V_{3}^{-1},
\end{align*}
where
\begin{align}\label{equ0041}
\widehat{X}=\bordermatrix{
~& n_{1}&n_{2} & n_{3}&n_{4}&n_{5}&q_{1}-r_{e} \cr
m_{1}&X_{11}&X_{12}&X_{13}&X_{14}&A_{15}&X_{16} \cr
m_{2}&X_{21}&X_{22}&A_{23}&A_{24}&A_{25}&X_{26} \cr
m_{3}&X_{31}&A_{32}&X_{33}&A_{34}-A_{36}&A_{35}&X_{36} \cr
m_{4}&X_{41}&A_{42}&A_{43}-A_{63}&A_{44}-A_{64}&A_{45}&X_{46} \cr
m_{5}&A_{51}&A_{52}&A_{53}&A_{54}&A_{55}&X_{56} \cr
p_{1}-r_{b}&X_{61}&X_{62}&X_{63}&X_{64}&X_{65}&X_{66}},
\end{align}
\begin{align}\label{equ0042}
\widehat{Y}=\bordermatrix{
~& n_{4}&n_{6} & n_{7}&n_{1}&n_{2}&q_{2}-r_{f} \cr
m_{4}&A_{66}-A_{64}&A_{67}-A_{47}&A_{68}&A_{61}-A_{41}+X_{41}&A_{62}&Y_{16} \cr
m_{6}&A_{76}-A_{74}&Y_{22}&A_{78}&Y_{24}&A_{72}&Y_{26} \cr
m_{7}&A_{86}&A_{87}&A_{88}&A_{81}&A_{82}&Y_{36} \cr
m_{1}&A_{16}-A_{14}+X_{14}&Y_{42}&A_{18}&Y_{44}&A_{12}-X_{12}&Y_{46} \cr
m_{2}&A_{26}&A_{27}&A_{28}&A_{21}-X_{21}&A_{22}-X_{22}&Y_{56} \cr
p_{2}-r_{c}&Y_{61}&Y_{62}&Y_{63}&Y_{64}&Y_{65}&Y_{66}},
\end{align}
\begin{align}\label{equ0043}
\widehat{Z}=\bordermatrix{
~& n_{8}&n_{4} & n_{6}&n_{3}&n_{1}&q_{3}-r_{g} \cr
m_{8}&A_{99}&A_{96}&A_{97}&A_{93}&A_{91}&Z_{16} \cr
m_{4}&A_{69}&A_{64}&A_{47}&A_{63}&A_{41}-X_{41}&Z_{26} \cr
m_{6}&A_{79}&A_{74}&A_{77}-Y_{22}&A_{73}&A_{71}-Y_{24}&Z_{36} \cr
m_{3}&A_{39}&A_{36}&A_{37}&A_{33}-X_{33}&A_{31}-X_{31}&Z_{46} \cr
m_{1}&A_{19}&A_{14}-X_{14}&A_{17}-Y_{42}&A_{13}-X_{13}&Z_{55}&Z_{56} \cr
p_{3}-r_{d}&Z_{61}&Z_{62}&Z_{63}&Z_{64}&Z_{65}&Z_{66}},
\end{align}
$A_{ij},T_{i},V_{i}$ are defined in Theorem \ref{theorem01},    the remaining $X_{ij},Y_{ij},Z_{ij}$ in (\ref{equ0041})-(\ref{equ0043}) are arbitrary matrices over $\mathbb{H}$
with appropriate sizes.
\end{theorem}

\begin{proof}
It follows from Theorem \ref{theorem01} that the matrix equation (\ref{system002}) is equivalent to the matrix equation
\begin{align}\label{equ0044}
S_{B}(T_{1}XV_{1})S_{E}+S_{C}(T_{2}YV_{2})S_{F}+S_{D}(T_{3}ZV_{3})S_{G}=S_{A}.
\end{align}
Let the matrices
\begin{align}\label{equ0045}
\widehat{X}=T_{1}XV_{1}=\begin{pmatrix}X_{11}&\cdots&X_{16}\\
\vdots&\ddots&\vdots\\
X_{61}&\cdots&X_{66}\end{pmatrix},
\end{align}
\begin{align}
\widehat{Y}=T_{2}YV_{2}=\begin{pmatrix}Y_{11}&\cdots&Y_{16}\\
\vdots&\ddots&\vdots\\
Y_{61}&\cdots&Y_{66}\end{pmatrix},
\end{align}
\begin{align}\label{equ0047}
\widehat{Z}=T_{3}ZV_{3}=\begin{pmatrix}Z_{11}&\cdots&Z_{16}\\
\vdots&\ddots&\vdots\\
Z_{61}&\cdots&Z_{66}\end{pmatrix},
\end{align}
be partitioned in accordance with (\ref{equ0044}). Then it follows from (\ref{equ0022})-(\ref{equ0024}) and (\ref{equ0044})-(\ref{equ0047}) that
\begin{align*}
\begin{pmatrix}
\scriptstyle  X_{11}+Y_{44}+Z_{55}&\scriptstyle   X_{12}+Y_{45}&\scriptstyle   X_{13}+Z_{54}&\scriptstyle   X_{14}+Z_{52}&\scriptstyle   X_{15} &\scriptstyle    Y_{41}+Z_{52} &\scriptstyle   Y_{42}+Z_{53}&\scriptstyle    Y_{43}&\scriptstyle    Z_{51} &\scriptstyle      0 &\scriptstyle    0 \\
\scriptstyle  X_{21}+Y_{54}&\scriptstyle   X_{22}+Y_{55}&\scriptstyle   X_{23}&\scriptstyle   X_{24}&\scriptstyle   X_{25} &\scriptstyle    Y_{51} &\scriptstyle   Y_{52}&\scriptstyle    Y_{53} &\scriptstyle    0 &\scriptstyle      0 &\scriptstyle    0 \\
\scriptstyle   X_{31}+Z_{45}&\scriptstyle   X_{32}&\scriptstyle   X_{33}+Z_{44}&\scriptstyle   X_{34}+Z_{42}&\scriptstyle   X_{35} &\scriptstyle    Z_{42} &\scriptstyle   Z_{43} &\scriptstyle   0 &\scriptstyle    Z_{41} &\scriptstyle     0&\scriptstyle    0 \\
\scriptstyle  X_{41}+Z_{25}&\scriptstyle   X_{42}&\scriptstyle   X_{43}+Z_{24}&\scriptstyle   X_{44}+Z_{22}&\scriptstyle   X_{45} &\scriptstyle    Z_{22}&\scriptstyle   Z_{23} &\scriptstyle   0&\scriptstyle    Z_{21} &\scriptstyle     0 &\scriptstyle    0 \\
\scriptstyle  X_{51}&\scriptstyle   X_{52}&\scriptstyle   X_{53}&\scriptstyle   X_{54}&\scriptstyle   X_{55} &\scriptstyle    0 &\scriptstyle   0 &\scriptstyle    0 &\scriptstyle    0 &\scriptstyle      0 &\scriptstyle    0\\
\scriptstyle  Y_{14}+Z_{25} &\scriptstyle    Y_{15} &\scriptstyle    Z_{24} &\scriptstyle   Z_{22} &\scriptstyle    0 &\scriptstyle    Y_{11} +Z_{22}&\scriptstyle   Y_{12} +Z_{23} &\scriptstyle   Y_{13}&\scriptstyle   Z_{21} &\scriptstyle      0 &\scriptstyle    0 \\
\scriptstyle  Y_{24}+Z_{35} &\scriptstyle    Y_{25} &\scriptstyle    Z_{34} &\scriptstyle   Z_{32} &\scriptstyle    0 &\scriptstyle    Y_{21} +Z_{32}&\scriptstyle   Y_{22}+Z_{33} &\scriptstyle   Y_{23} &\scriptstyle     Z_{31}&\scriptstyle     0 &\scriptstyle    0 \\
\scriptstyle  Y_{34} &\scriptstyle   Y_{35}&\scriptstyle    0 &\scriptstyle   0 &\scriptstyle    0 &\scriptstyle    Y_{31} &\scriptstyle   Y_{32} &\scriptstyle   Y_{33} &\scriptstyle    0 &\scriptstyle     0 &\scriptstyle    0\\
\scriptstyle   Z_{15} &\scriptstyle    0&\scriptstyle    Z_{14} &\scriptstyle   Z_{12} &\scriptstyle    0 &\scriptstyle    Z_{12} &\scriptstyle   Z_{13} &\scriptstyle    0 &\scriptstyle    Z_{11} &\scriptstyle      0 &\scriptstyle    0 \\
\scriptstyle   0 &\scriptstyle    0 &\scriptstyle   0 &\scriptstyle   0 &\scriptstyle   0 &\scriptstyle    0 &\scriptstyle   0 &\scriptstyle    0 &\scriptstyle   0 &\scriptstyle      0 &\scriptstyle   0\\
\scriptstyle  0 &\scriptstyle    0 &\scriptstyle    0 &\scriptstyle    0 &\scriptstyle      0 &\scriptstyle    0 &\scriptstyle    0&\scriptstyle    0 &\scriptstyle    0 &\scriptstyle    0 &\scriptstyle     0
\end{pmatrix}
\end{align*}
\begin{align}\label{equ0048}
=\begin{pmatrix}
A_{11} & A_{12} & A_{13} &A_{14} & A_{15} & A_{16} &A_{17} & A_{18} & A_{19} &   A_{1,10} & 0 \\
A_{21} & A_{22} & A_{23} &A_{24} & A_{25} & A_{26} &A_{27} & A_{28} & A_{29} &   A_{2,10} & 0 \\
 A_{31} & A_{32} & A_{33} &A_{34} & A_{35} & A_{36} &A_{37} & A_{38} & A_{39} &   A_{3,10} & 0 \\
A_{41} & A_{42} & A_{43} &A_{44} & A_{45} & A_{46} &A_{47} & A_{48} & A_{49} &   A_{4,10} & 0 \\
 A_{51} & A_{52} & A_{53} &A_{54} & A_{55} & A_{56} &A_{57} & A_{58} & A_{59} &   A_{5,10} & 0\\
A_{61} & A_{62} & A_{63} &A_{64} & A_{65} & A_{66} &A_{67} & A_{68} & A_{69} &   A_{6,10} & 0 \\
A_{71} & A_{72} & A_{73} &A_{74} & A_{75} & A_{76} &A_{77} & A_{78} & A_{79} &   A_{7,10} & 0 \\
A_{81} & A_{82} & A_{83} &A_{84} & A_{85} & A_{86} &A_{87} & A_{88} & A_{89} &   A_{8,10} & 0\\
 A_{91} & A_{92} & A_{93} &A_{94} & A_{95} & A_{96} &A_{97} & A_{98} & A_{99} &   A_{9,10} & 0 \\
 A_{10,1} & A_{10,2} & A_{10,3} &A_{10,4} & A_{10,5} & A_{10,6} &A_{10,7} & A_{10,8} & A_{10,9} &   0 &0\\
0 & 0 & 0 & 0 &   0 & 0 & 0& 0 & 0 & 0 &  I_{t}\end{pmatrix}.
\end{align}

If the equation (\ref{system002}) has a solution $(X,Y,Z)$, then by (\ref{equ0048}), we have that the equalities in
(\ref{equ00401}) hold, and
\begin{align*}
&X_{11}+Y_{44}+Z_{55}=A_{11},~X_{12}+Y_{45}=A_{12},~X_{13}+Z_{54}=A_{13},~X_{14}+Z_{52}=A_{14},~X_{15}=A_{15}, \\
&Y_{41}+Z_{52}=A_{16},~Y_{42}+Z_{53}=A_{17},~ Y_{43}=A_{18},~ Z_{51}=A_{19},X_{21}+Y_{54}=A_{21},~X_{22}+Y_{55}=A_{22},\\
&X_{23}=A_{23},~X_{24}=A_{24},~X_{25}=A_{25},~Y_{51}=A_{26},~Y_{52}=A_{27},~Y_{53}=A_{28},~X_{31}+Z_{45}=A_{31}, \\
&X_{32}=A_{32}, ~X_{33}+Z_{44}=A_{33}, ~X_{34}+Z_{42}=A_{34}, ~X_{35}=A_{35}, ~ Z_{42}=A_{36}, ~Z_{43}=A_{37}, ~ Z_{41}=A_{39},\\
&X_{41}+Z_{25}=A_{41}, ~X_{42}=A_{42}, ~X_{43}+Z_{24}=A_{43}, ~X_{44}+Z_{22}=A_{44}, ~X_{45}=A_{45}, ~ Z_{22}=A_{46}, \\
&Z_{23}=A_{47},  ~ Z_{21}=A_{49}, ~X_{51}=A_{51}, ~X_{52}=A_{52}, ~X_{53}=A_{53}, ~X_{54}=A_{54}, ~X_{55}=A_{55}\\
&Y_{14}+Z_{25}=A_{61}, ~ Y_{15}=A_{62}, ~ Z_{24}=A_{63}, ~Z_{22}=A_{64},  ~ Y_{11}+Z_{22}=A_{66}, ~Y_{12}+Z_{23}=A_{67}, \\
& Y_{13}=A_{68}, ~Z_{21}=A_{69},  ~Y_{24}+Z_{35}=A_{71}, ~ Y_{25}=A_{72}, ~ Z_{34}=A_{73}, ~Z_{32}=A_{74},  ~ Y_{21} +Z_{32}=A_{76}, \\
& Y_{22}+Z_{33}=A_{77}, ~Y_{23}=A_{78} , ~  Z_{31}=A_{79},  ~Y_{34}=A_{81}, ~Y_{35}=A_{82}, ~ Y_{31}=A_{86}, ~Y_{32}=A_{87}, \\
& Y_{33}=A_{88},~Z_{15}=A_{91},  ~ Z_{14}=A_{93}, ~Z_{12}=A_{94},  ~ Z_{12}=A_{96}, ~Z_{13}=A_{97},   ~ Z_{11}=A_{99}.
\end{align*}
Hence, $(X,Y,Z)$ can be expressed as (\ref{equ0041})-(\ref{equ0043}) by (\ref{equ0045})-(\ref{equ0047}).

Conversely, assume that the equalities in (\ref{equ00401}) hold, then by (\ref{equ0022})-(\ref{equ0024}) and  (\ref{equ0044})-(\ref{equ0048}), it can be
verified that the matrices have the forms of (\ref{equ0041})-(\ref{equ0043}) is a solution of (\ref{equ0044}), i.e., (\ref{system002}).

\end{proof}

\begin{remark}
In our opinion the presented expression of the general solution  is more useful   than the expression found by Wang et al. \cite{QWWangandyushaowen}, since the latter can not be used to consider the maximal and minimal ranks of the general solution to (\ref{system002}).
\end{remark}

\subsection{\textbf{The range of ranks  of the   general
solution to (\ref{system002})}}

In this section, we consider the maximal and minimal ranks of the general
solution to the real quaternion matrix equation (\ref{system002}).

\begin{theorem}Let $A\in \mathbb{H}^{m\times n},
B\in \mathbb{H}^{m\times p_{1}},C\in \mathbb{H}^{m\times p_{2}},D\in \mathbb{H}^{m\times p_{3}},E\in \mathbb{H}^{q_{1}\times n},F\in \mathbb{H}^{q_{2}\times n}$  and $G\in \mathbb{H}^{q_{3}\times n}$ be given. Assume that  equation (\ref{system002}) is consistent. Then,
\begin{align*}
\mathop {\max }\limits_{BXE+CYF+DZG=A } r \left({Z }\right)
 =\min\left\{ p_{3},~q_{3},~p_{3}+q_{3}+r_{a|e|f}-r_{d}-r_{e|f|g},~p_{3}+q_{3}+r_{abc}-r_{g}-r_{bcd},\right.
\end{align*}
\begin{align*}
p_{3}+q_{3}+r_{d0ab0|da00c|0e000|00f00}-r_{gg|e0|0f}-r_{bd}-r_{cd},~p_{3}+q_{3}+r_{ac|e0}-r_{cd}-r_{e|f},
\end{align*}
\begin{align*}
\left. p_{3}+q_{3}+r_{ab|f0}-r_{f|g}-r_{bd},~p_{3}+q_{3}+r_{gg00|0ab0|a00c|e000|0f00}-r_{db0|d0c}-r_{e|g}-r_{f|g}\right\},
\end{align*}
\begin{align*}
&\mathop {\min }\limits_{ BXE+CYF+DZG=A } r \left({Z }\right)
\\=&r_{d0ab0|da00c|0e000|00f00}+r_{gg00|0ab0|a00c|e000|0f00}+r_{abc}+r_{a|e|f}+r_{bd}+r_{e|g}+r_{bc}+r_{e|f}\\&-r_{d0ab00|da00cb|0e0000|00f000}
-r_{gg00|0ab0|a00c|e000|0f00|0e00}
\\&+ \max\big\{ r_{ac|e0}-r_{d0ab0|da00c|0e000|00f00|00e00}-r_{gg000|0ab0c|a00c0|e0000|0f000},\\&
\qquad \qquad r_{ab|f0}-r_{d0ab0|da00c|0e000|00f00|0f000}-r_{gg000|0ab00|a00cb|e0000|0f000}\big\}.
\end{align*}
\end{theorem}

\begin{proof}
It follows from  Theorem \ref{theorem04} that the expression of $Z$ in (\ref{system002}) can be expressed as $Z=T_{3}^{-1}\widehat{Z}V_{3}^{-1},$
where $\widehat{Z}$ is given in (\ref{equ0043}). Clearly, $r(Z)=r(T_{3}^{-1}\widehat{Z}V_{3}^{-1})=r(\widehat{Z}).$ Now we consider the
maximal and minimal ranks of $\widehat{Z}$. Applying Lemma \ref{lemma01} to the variable matrices $(Z_{61},~Z_{62},~Z_{63},~Z_{64},~Z_{65})$ and $\begin{pmatrix}\begin{smallmatrix}Z_{16}\\Z_{26}\\Z_{36}\\Z_{46}\\Z_{56}\end{smallmatrix}\end{pmatrix}$ of $\widehat{Z}$, we obtain
\begin{align*}
\mathop {\max }\limits_{ Z_{i6},~Z_{6i},~i=1,\ldots,5 } r ({\widehat{Z }})
 =\min\left\{ p_{3},q_{3},p_{3}+q_{3}-r_{d}-r_{g}+r(\Theta),r(Z_{66})+r_{d}+r_{g}\right\},
\end{align*}
\begin{align*}
\mathop {\min }\limits_{ Z_{i6},~Z_{6i},~i=1,\ldots,5 } r ({\widehat{Z} })
 = \max\left\{ r(\Theta),r(Z_{66})\right\},
\end{align*}
where
\begin{align*}
\Theta=\bordermatrix{
~& n_{8}&n_{4} & n_{6}&n_{3}&n_{1} \cr
m_{8}&A_{99}&A_{96}&A_{97}&A_{93}&A_{91}  \cr
m_{4}&A_{69}&A_{64}&A_{47}&A_{63}&A_{41}-X_{41} \cr
m_{6}&A_{79}&A_{74}&A_{77}-Y_{22}&A_{73}&A_{71}-Y_{24} \cr
m_{3}&A_{39}&A_{36}&A_{37}&A_{33}-X_{33}&A_{31}-X_{31}  \cr
m_{1}&A_{19}&A_{14}-X_{14}&A_{17}-Y_{42}&A_{13}-X_{13}&Z_{55} }.
\end{align*}
Note that
$$\mathop {\max } r\left( {Z_{66}} \right)=\mathop {\min } \left\{p_{3}-r_{d},q_{3}-r_{g}\right\},~
\mathop {\max } r\left( {Z_{66}} \right)=0.$$
Hence, we have
\begin{align}
\mathop {\max }\limits_{ Z_{i6},~Z_{6i},~i=1,\ldots,6 } r ({\widehat{Z }})
 =\min\left\{ p_{3},q_{3},p_{3}+q_{3}-r_{d}-r_{g}+r(\Theta)\right\},
~
\mathop {\min }\limits_{ Z_{i6},~Z_{6i},~i=1,\ldots,6 } r ({\widehat{Z} })
 = r(\Theta).
\end{align}
Applying Lemma \ref{lemma03} to the variable matrices $\begin{pmatrix} A_{14}-X_{14},&A_{17}-Y_{42},&A_{13}-X_{13},&Z_{55} \end{pmatrix}$ of $\Theta$, we obtain
\begin{align*}
\mathop {\max }\limits_{ \begin{pmatrix}\begin{smallmatrix}A_{14}-X_{14},&A_{17}-Y_{42},&A_{13}-X_{13},&Z_{55}\end{smallmatrix}\end{pmatrix} } r \left({\Theta }\right)
 =\min\left\{ m_{1}+r(\Theta_{1}),n_{1}+n_{3}+n_{4}+n_{6}+r\begin{pmatrix}A_{99}\\
A_{69}\\
A_{79}\\
A_{39}\\
A_{19}\end{pmatrix}\right\},
\end{align*}
\begin{align*}
\mathop {\min }\limits_{ \begin{pmatrix}\begin{smallmatrix}A_{14}-X_{14},&A_{17}-Y_{42},&A_{13}-X_{13},&Z_{55}\end{smallmatrix}\end{pmatrix} } r \left({\Theta }\right)
 = r(\Theta_{1})+r\begin{pmatrix}A_{99}\\
A_{69}\\
A_{79}\\
A_{39}\\
A_{19}\end{pmatrix}-r\begin{pmatrix}A_{99}\\
A_{69}\\
A_{79}\\
A_{39}\end{pmatrix},
\end{align*}
where
\begin{align*}
\Theta_{1}=\bordermatrix{
~& n_{8}&n_{4} & n_{6}&n_{3}&n_{1} \cr
m_{8}&A_{99}&A_{96}&A_{97}&A_{93}&A_{91}  \cr
m_{4}&A_{69}&A_{64}&A_{47}&A_{63}&A_{41}-X_{41} \cr
m_{6}&A_{79}&A_{74}&A_{77}-Y_{22}&A_{73}&A_{71}-Y_{24} \cr
m_{3}&A_{39}&A_{36}&A_{37}&A_{33}-X_{33}&A_{31}-X_{31}   }.
\end{align*}
Applying Lemma \ref{lemma03} to the variable matrices $\begin{pmatrix}\begin{smallmatrix}A_{41}-X_{41}\\A_{71}-Y_{24}\\A_{31}-X_{31}\end{smallmatrix}\end{pmatrix}$ of $\Theta_{1}$, we obtain
\begin{align*}
\mathop {\max }\limits_{ \begin{pmatrix}\begin{smallmatrix}A_{41}-X_{41}\\A_{71}-Y_{24}\\A_{31}-X_{31}\end{smallmatrix}\end{pmatrix} } r \left({\Theta_{1} }\right)
 =\min\left\{ m_{3}+m_{4}+m_{6}+r\begin{pmatrix}A_{99},&A_{96},&A_{97},&A_{93},&A_{91}\end{pmatrix},n_{1}+r(\Theta_{2})\right\},
\end{align*}
\begin{align*}
\mathop {\min }\limits_{ \begin{pmatrix}\begin{smallmatrix}A_{41}-X_{41}\\A_{71}-Y_{24}\\A_{31}-X_{31}\end{smallmatrix}\end{pmatrix} } r \left({\Theta_{1} }\right)
 = r(\Theta_{2})+r\begin{pmatrix}A_{99},&A_{96},&A_{97},&A_{93},&A_{91}\end{pmatrix}-r\begin{pmatrix}A_{99},&A_{96},&A_{97},&A_{93}\end{pmatrix},
\end{align*}
where
\begin{align*}
\Theta_{1}=\bordermatrix{
~& n_{8}&n_{4} & n_{6}&n_{3}  \cr
m_{8}&A_{99}&A_{96}&A_{97}&A_{93}  \cr
m_{4}&A_{69}&A_{64}&A_{47}&A_{63}  \cr
m_{6}&A_{79}&A_{74}&A_{77}-Y_{22}&A_{73}  \cr
m_{3}&A_{39}&A_{36}&A_{37}&A_{33}-X_{33}   }.
\end{align*}
Applying Lemma \ref{lemma04} to the variable matrices $A_{77}-Y_{22}$ and $A_{33}-X_{33}$ of $\Theta_{2}$, we obtain
\begin{align*}
&\mathop {\max }\limits_{ A_{77}-Y_{22},~A_{33}-X_{33} } r \left({\Theta_{2} }\right)
 \nonumber\\=&\min\left\{ m_{3}+m_{6}+r\begin{pmatrix}A_{99}&A_{96}&A_{97}&A_{93}\\A_{69}&A_{64}&A_{47}&A_{63}\end{pmatrix},
 n_{3}+m_{6}+r\begin{pmatrix}A_{99}&A_{96}&A_{97}\\A_{69}&A_{64}&A_{47}\\A_{39}&A_{36}&A_{37}\end{pmatrix},\right.
\end{align*}
\begin{align}
\left.  m_{3}+n_{6}+r\begin{pmatrix}A_{99}&A_{96}&A_{93}\\A_{69}&A_{64}&A_{63}\\A_{79}&A_{74}&A_{73}\end{pmatrix},
n_{3}+n_{6}+r\begin{pmatrix}A_{99}&A_{96} \\A_{69}&A_{64} \\A_{79}&A_{74}\\A_{39}&A_{36}\end{pmatrix}
\right\}.
\end{align}
\begin{align*}
&\mathop {\min }\limits_{ A_{77}-Y_{22},~A_{33}-X_{33} } r \left({\Theta_{2} }\right)
  =r\begin{pmatrix}A_{99}&A_{96}&A_{97}&A_{93}\\A_{69}&A_{64}&A_{47}&A_{63}\end{pmatrix}+r\begin{pmatrix}A_{99}&A_{96} \\A_{69}&A_{64} \\A_{79}&A_{74}\\A_{39}&A_{36}\end{pmatrix}\\&+\max\left\{ r\begin{pmatrix}A_{99}&A_{96}&A_{97}\\A_{69}&A_{64}&A_{47}\\A_{39}&A_{36}&A_{37}\end{pmatrix}
  -r\begin{pmatrix}A_{99}&A_{96}&A_{97}\\A_{69}&A_{64}&A_{47} \end{pmatrix}-r\begin{pmatrix}A_{99}&A_{96} \\A_{69}&A_{64} \\A_{39}&A_{36} \end{pmatrix},\right.
\end{align*}
\begin{align*}
\left.   r\begin{pmatrix}A_{99}&A_{96}&A_{93}\\A_{69}&A_{64}&A_{63}\\A_{79}&A_{74}&A_{73}\end{pmatrix}-
 r\begin{pmatrix}A_{99}&A_{96}&A_{93}\\A_{69}&A_{64}&A_{63} \end{pmatrix}-
  r\begin{pmatrix}A_{99}&A_{96} \\A_{69}&A_{64} \\A_{79}&A_{74} \end{pmatrix}
\right\}.
\end{align*}
Hence,
\begin{align*}
\mathop {\max }\limits_{BXE+CYF+DZG=A } r \left({Z }\right)
 =\min\left\{ p_{3},q_{3},s_{1},s_{2},s_{3},s_{4},s_{5},s_{6}\right\},
\end{align*}
\begin{align*}
\mathop {\min }\limits_{BXE+CYF+DZG=A } r \left({Z }\right)
 =\max\left\{ s_{7},s_{8}\right\},
\end{align*}
where
\begin{align*}
s_{1}=p_{3}+q_{3}-r_{d}-r_{g}+n_{1}+n_{3}+n_{4}+n_{6}+r\begin{pmatrix}A_{99}\\
A_{69}\\
A_{79}\\
A_{39}\\
A_{19}\end{pmatrix},
\end{align*}
\begin{align*}
s_{2}=p_{3}+q_{3}-r_{d}-r_{g}+m_{1}+m_{3}+m_{4}+m_{6}+r\begin{pmatrix}A_{99},&A_{96},&A_{97},&A_{93},&A_{91}\end{pmatrix},
\end{align*}
\begin{align*}
s_{3}=p_{3}+q_{3}-r_{d}-r_{g}+m_{1}+n_{1}+m_{3}+m_{6}+r\begin{pmatrix}A_{99}&A_{96}&A_{97}&A_{93}\\A_{69}&A_{64}&A_{47}&A_{63}\end{pmatrix},
\end{align*}
\begin{align*}
s_{4}=p_{3}+q_{3}-r_{d}-r_{g}+m_{1}+n_{1}+ n_{3}+m_{6}+r\begin{pmatrix}A_{99}&A_{96}&A_{97}\\A_{69}&A_{64}&A_{47}\\A_{39}&A_{36}&A_{37}\end{pmatrix},
\end{align*}
\begin{align*}
s_{5}=p_{3}+q_{3}-r_{d}-r_{g}+m_{1}+n_{1}+m_{3}+n_{6}+r\begin{pmatrix}A_{99}&A_{96}&A_{93}\\A_{69}&A_{64}&A_{63}\\A_{79}&A_{74}&A_{73}\end{pmatrix},
\end{align*}
\begin{align*}
s_{6}=p_{3}+q_{3}-r_{d}-r_{g}+m_{1}+n_{1}+n_{3}+n_{6}+r\begin{pmatrix}A_{99}&A_{96} \\A_{69}&A_{64} \\A_{79}&A_{74}\\A_{39}&A_{36}\end{pmatrix},
\end{align*}
\begin{align*}
s_{7}= &r\begin{pmatrix}A_{99}&A_{96}&A_{97}&A_{93}\\A_{69}&A_{64}&A_{47}&A_{63}\end{pmatrix}+r\begin{pmatrix}A_{99}&A_{96} \\A_{69}&A_{64} \\A_{79}&A_{74}\\A_{39}&A_{36}\end{pmatrix}+r\begin{pmatrix}A_{99}\\
A_{69}\\
A_{79}\\
A_{39}\\
A_{19}\end{pmatrix}-r\begin{pmatrix}A_{99}\\
A_{69}\\
A_{79}\\
A_{39}\end{pmatrix}\\&+r\begin{pmatrix}A_{99},&A_{96},&A_{97},&A_{93},&A_{91}\end{pmatrix}-r\begin{pmatrix}A_{99},&A_{96},&A_{97},&A_{93}\end{pmatrix}\\&+ r\begin{pmatrix}A_{99}&A_{96}&A_{97}\\A_{69}&A_{64}&A_{47}\\A_{39}&A_{36}&A_{37}\end{pmatrix}
  -r\begin{pmatrix}A_{99}&A_{96}&A_{97}\\A_{69}&A_{64}&A_{47} \end{pmatrix}-r\begin{pmatrix}A_{99}&A_{96} \\A_{69}&A_{64} \\A_{39}&A_{36} \end{pmatrix},
\end{align*}
\begin{align*}
s_{8}=& r\begin{pmatrix}A_{99}&A_{96}&A_{97}&A_{93}\\A_{69}&A_{64}&A_{47}&A_{63}\end{pmatrix}+r\begin{pmatrix}A_{99}&A_{96} \\A_{69}&A_{64} \\A_{79}&A_{74}\\A_{39}&A_{36}\end{pmatrix}+r\begin{pmatrix}A_{99}\\
A_{69}\\
A_{79}\\
A_{39}\\
A_{19}\end{pmatrix}-r\begin{pmatrix}A_{99}\\
A_{69}\\
A_{79}\\
A_{39}\end{pmatrix}\\&+r\begin{pmatrix}A_{99},&A_{96},&A_{97},&A_{93},&A_{91}\end{pmatrix}-r\begin{pmatrix}A_{99},&A_{96},&A_{97},&A_{93}\end{pmatrix}\\&+
  r\begin{pmatrix}A_{99}&A_{96}&A_{93}\\A_{69}&A_{64}&A_{63}\\A_{79}&A_{74}&A_{73}\end{pmatrix}-
 r\begin{pmatrix}A_{99}&A_{96}&A_{93}\\A_{69}&A_{64}&A_{63} \end{pmatrix}-
  r\begin{pmatrix}A_{99}&A_{96} \\A_{69}&A_{64} \\A_{79}&A_{74} \end{pmatrix}.
\end{align*}
Now we pay attention to the ranks of the block matrices in $s_{i}$. Upon construction and computation, we obtain
\begin{align}\label{equh514}
r\begin{pmatrix}A_{99}&A_{96} \\A_{69}&A_{64} \\A_{79}&A_{74}\\A_{39}&A_{36}\end{pmatrix}=&r\begin{pmatrix}S_{G}&S_{G}&0&0\\0&S_{A}&S_{B}&0\\S_{A}&0&0&S_{C}\\S_{E}&0&0&0\\0&S_{F}&0&0\end{pmatrix}
-r(S_{E})-r(S_{F})-r(S_{B})-r(S_{C})\nonumber\\&-n_{3}-n_{4}-n_{6}-n_{8}\nonumber\\
=&r_{gg00|0ab0|a00c|e000|0f00}-r_{e}-r_{f}-r_{b}-r_{c}-n_{3}-n_{4}-n_{6}-n_{8},
\end{align}
\begin{align}
r\begin{pmatrix}A_{99}&A_{96}&A_{97}&A_{93}\\A_{69}&A_{64}&A_{47}&A_{63}\end{pmatrix}
=&r\begin{pmatrix}S_{D}&0&S_{A}&S_{B}&0\\S_{D}&S_{A}&0&0&S_{C}\\0&S_{E}&0&0&0\\0&0&S_{F}&0&0\end{pmatrix}
-r(S_{E})-r(S_{F})-r(S_{B})-r(S_{C})\nonumber\\&-m_{3}-m_{4}-m_{6}-m_{8}\nonumber\\=&
r_{d0ab0|da00c|0e000|00f00}-r_{e}-r_{f}-r_{b}-r_{c}-m_{3}-m_{4}-m_{6}-m_{8},
\end{align}
\begin{align}
r(A_{99},~A_{96},~A_{97},~A_{93},~A_{91})=
r(S_{A},~S_{B},~S_{C})-r(S_{B},~S_{C})=
r_{abc}-r_{bc},
\end{align}
\begin{align}
r\begin{pmatrix}A_{99}\\
A_{69}\\
A_{79}\\
A_{39}\\
A_{19}\end{pmatrix}=r\begin{pmatrix}S_{A}\\ S_{E}\\ S_{F}\end{pmatrix}-r\begin{pmatrix}S_{E}\\ S_{F}\end{pmatrix}
=r_{a|e|f}-r_{e|f},
\end{align}
\begin{align}
r\begin{pmatrix}
A_{99}\\
A_{69}\\
A_{79}\\
A_{39}\end{pmatrix}
=&r\begin{pmatrix}S_{G}&S_{G}&0&0\\0&S_{A}&S_{B}&0\\S_{A}&0&0&S_{C}\\S_{E}&0&0&0\\0&S_{F}&0&0\\0&S_{E}&0&0\end{pmatrix}
-r\begin{pmatrix}S_{E}\\S_{F}\end{pmatrix}-r(S_{E})-r(S_{B})-r(S_{C})\nonumber\\&-n_{4}-n_{6}-n_{8}\nonumber\\
=&r_{gg00|0ab0|a00c|e000|0f00|0e00}-r_{e|f}-r_{e}-r_{b}-r_{c}-n_{4}-n_{6}-n_{8},
\end{align}
\begin{align}
r\begin{pmatrix}A_{99},&A_{96},&A_{97},&A_{93}\end{pmatrix}=&
r\begin{pmatrix}S_{D}&0&S_{A}&S_{B}&0&0\\S_{D}&S_{A}&0&0&S_{C}&S_{B}\\0&S_{E}&0&0&0&0\\0&0&S_{F}&0&0&0\end{pmatrix}
-r(S_{E})-r(S_{F})-r(S_{B})\nonumber\\&-r(S_{B},~S_{C})-m_{4}-m_{6}-m_{8}\nonumber\\=&
r_{d0ab00|da00cb|0e0000|00f000}-r_{e}-r_{f}-r_{b}-r_{bc}-m_{4}-m_{6}-m_{8},
\end{align}
\begin{align}
r\begin{pmatrix}A_{99}&A_{96}&A_{93}\\A_{69}&A_{64}&A_{63}\\A_{79}&A_{74}&A_{73}\end{pmatrix}
=r\begin{pmatrix}S_{A}&S_{B}\\S_{F}&0\end{pmatrix}-r(S_{B})-r(S_{F})=r_{ab|f0}-r_{b}-r_{f},
\end{align}
\begin{align}
r\begin{pmatrix}A_{99}&A_{96}&A_{97}\\A_{69}&A_{64}&A_{47}\\A_{39}&A_{36}&A_{37}\end{pmatrix}
=r\begin{pmatrix}S_{A}&S_{C}\\S_{E}&0\end{pmatrix}-r(S_{E})-r(S_{C})=r_{ac|e0}-r_{e}-r_{c},
\end{align}
\begin{align}
r\begin{pmatrix}A_{99}&A_{96} \\A_{69}&A_{64} \\A_{39}&A_{36} \end{pmatrix}
=&r\begin{pmatrix}S_{G}&S_{G}&0&0&0\\0&S_{A}&S_{B}&0&S_{C}\\S_{A}&0&0&S_{C}&0\\S_{E}&0&0&0&0\\0&S_{F}&0&0&0\end{pmatrix}
-r(S_{E})-r(S_{F})-r(S_{C})-r(S_{B},~S_{C})\nonumber\\&-n_{3}-n_{4}-n_{6}-n_{8}\nonumber\\
=&r_{gg000|0ab0c|a00c0|e0000|0f000}-r_{e}-r_{f}-r_{c}-r_{bc}-n_{3}-n_{4}-n_{6}-n_{8},
\end{align}
\begin{align}
r\begin{pmatrix}A_{99}&A_{96} \\A_{69}&A_{64} \\A_{79}&A_{74} \end{pmatrix}
=&r\begin{pmatrix}S_{G}&S_{G}&0&0&0\\0&S_{A}&S_{B}&0&0\\S_{A}&0&0&S_{C}&S_{B}\\S_{E}&0&0&0&0\\0&S_{F}&0&0&0\end{pmatrix}
-r(S_{E})-r(S_{F})-r(S_{B})-r(S_{B},~S_{C})\nonumber\\&-n_{3}-n_{4}-n_{6}-n_{8}\nonumber\\
=&r_{gg000|0ab00|a00cb|e0000|0f000}-r_{e}-r_{f}-r_{b}-r_{bc}-n_{3}-n_{4}-n_{6}-n_{8},
\end{align}
\begin{align}
r\begin{pmatrix}A_{99}&A_{96}&A_{97}\\A_{69}&A_{64}&A_{47} \end{pmatrix}
=&r\begin{pmatrix}S_{D}&0&S_{A}&S_{B}&0\\S_{D}&S_{A}&0&0&S_{C}\\0&S_{E}&0&0&0\\0&0&S_{F}&0&0\\0&0&S_{E}&0&0\end{pmatrix}
-r(S_{B})-r(S_{C})-r(S_{E})-r\begin{pmatrix}S_{E}\\S_{F}\end{pmatrix}\nonumber\\&
-m_{3}-m_{4}-m_{6}-m_{8}\nonumber\\
=&r_{d0ab0|da00c|0e000|00f00|00e00}-r_{b}-r_{c}-r_{e}-r_{e|f}-m_{3}-m_{4}-m_{6}-m_{8},
\end{align}
\begin{align}\label{equh523}
r\begin{pmatrix}A_{99}&A_{96}&A_{93}\\A_{69}&A_{64}&A_{63} \end{pmatrix}
=&r\begin{pmatrix}S_{D}&0&S_{A}&S_{B}&0\\S_{D}&S_{A}&0&0&S_{C}\\0&S_{E}&0&0&0\\0&0&S_{F}&0&0\\0&S_{F}&0&0&0\end{pmatrix}
-r(S_{B})-r(S_{C})-r(S_{F})-r\begin{pmatrix}S_{E}\\S_{F}\end{pmatrix}\nonumber\\&-m_{3}-m_{4}-m_{6}-m_{8}\nonumber\\
=&r_{d0ab0|da00c|0e000|00f00|0f000}-r_{b}-r_{c}-r_{f}-r_{e|f}-m_{3}-m_{4}-m_{6}-m_{8}.
\end{align}
Hence from  (\ref{equh025})-(\ref{equh0210}) and (\ref{equh514})-(\ref{equh523}), we deduce that
\begin{align*}
s_{1}=p_{3}+q_{3}-r_{d}-r_{g}+n_{1}+n_{3}+n_{4}+n_{6}+r\begin{pmatrix}A_{99}\\
A_{69}\\
A_{79}\\
A_{39}\\
A_{19}\end{pmatrix}=p_{3}+q_{3}+r_{a|e|f}-r_{d}-r_{e|f|g},
\end{align*}
\begin{align*}
s_{2}=&p_{3}+q_{3}-r_{d}-r_{g}+m_{1}+m_{3}+m_{4}+m_{6}+r\begin{pmatrix}A_{99},&A_{96},&A_{97},&A_{93},&A_{91}\end{pmatrix}\\=&
p_{3}+q_{3}+r_{abc}-r_{g}-r_{bcd},
\end{align*}
\begin{align*}
s_{3}=&p_{3}+q_{3}-r_{d}-r_{g}+m_{1}+n_{1}+m_{3}+m_{6}+r\begin{pmatrix}A_{99}&A_{96}&A_{97}&A_{93}\\A_{69}&A_{64}&A_{47}&A_{63}\end{pmatrix}\\=&
p_{3}+q_{3}-r_{b}-r_{c}-r_{d}-r_{e}-r_{f}-r_{g}+r_{d0ab0|da00c|0e000|00f00}+n_{1}+m_{1}-m_{4}-m_{8}\\=&
p_{3}+q_{3}+r_{d0ab0|da00c|0e000|00f00}-r_{gg|e|f}-r_{bd}-r_{cd},
\end{align*}
\begin{align*}
s_{4}=&p_{3}+q_{3}-r_{d}-r_{g}+m_{1}+n_{1}+ n_{3}+m_{6}+r\begin{pmatrix}A_{99}&A_{96}&A_{97}\\A_{69}&A_{64}&A_{47}\\A_{39}&A_{36}&A_{37}\end{pmatrix}
\\=&p_{3}+q_{3}-r_{d}-r_{g}+m_{1}+n_{1}+ n_{3}+m_{6}+r_{ac|e0}-r_{e}-r_{c}
=p_{3}+q_{3}+r_{ac|e0}-r_{cd}-r_{e|f},
\end{align*}
\begin{align*}
s_{5}=&p_{3}+q_{3}-r_{d}-r_{g}+m_{1}+n_{1}+m_{3}+n_{6}+r\begin{pmatrix}A_{99}&A_{96}&A_{93}\\A_{69}&A_{64}&A_{63}\\A_{79}&A_{74}&A_{73}\end{pmatrix}\\
=&p_{3}+q_{3}-r_{d}-r_{g}+m_{1}+n_{1}+m_{3}+n_{6}+r_{ab|f0}-r_{b}-r_{f}=p_{3}+q_{3}+r_{ab|f0}-r_{f|g}-r_{bd},
\end{align*}
\begin{align*}
s_{6}=&p_{3}+q_{3}-r_{d}-r_{g}+m_{1}+n_{1}+n_{3}+n_{6}+r\begin{pmatrix}A_{99}&A_{96} \\A_{69}&A_{64} \\A_{79}&A_{74}\\A_{39}&A_{36}\end{pmatrix}\\=&
p_{3}+q_{3}-r_{b}-r_{c}-r_{d}-r_{e}-r_{f}-r_{g}+r_{gg00|0ab0|a00c|e000|0f00}+m_{1}+n_{1}-n_{4}-n_{8}\\=&
p_{3}+q_{3}+r_{gg00|0ab0|a00c|e000|0f00}-r_{db0|d0c}-r_{e|g}-r_{f|g},
\end{align*}
\begin{align*}
s_{7}=&r_{d0ab0|da00c|0e000|00f00}+r_{gg00|0ab0|a00c|e000|0f00}+r_{abc}+r_{a|e|f}+r_{bd}+r_{e|g}+r_{bc}+r_{e|f}\\&-r_{d0ab00|da00cb|0e0000|00f000}
-r_{gg00|0ab0|a00c|e000|0f00|0e00}
\\&+r_{ac|e0}-r_{d0ab0|da00c|0e000|00f00|00e00}-r_{gg000|0ab0c|a00c0|e0000|0f000},
\end{align*}
\begin{align*}
s_{8}=&r_{d0ab0|da00c|0e000|00f00}+r_{gg00|0ab0|a00c|e000|0f00}+r_{abc}+r_{a|e|f}+r_{bd}+r_{e|g}+r_{bc}+r_{e|f}\\&-r_{d0ab00|da00cb|0e0000|00f000}
-r_{gg00|0ab0|a00c|e000|0f00|0e00}
\\&+ r_{ab|f0}-r_{d0ab0|da00c|0e000|00f00|0f000}-r_{gg000|0ab00|a00cb|e0000|0f000}.
\end{align*}
\end{proof}

\begin{remark}
It is hard to derive the maximal and minimal ranks of the general solution to  (\ref{system002}) if we do not know the values of
$m_{i},n_{i},(i=1,2,3,4,6)$.

\end{remark}

Similarly, we can get the corresponding results on $X$ and $Y$. The proof is omitted.

\begin{theorem}Let $A\in \mathbb{H}^{m\times n},
B\in \mathbb{H}^{m\times p_{1}},C\in \mathbb{H}^{m\times p_{2}},D\in \mathbb{H}^{m\times p_{3}},E\in \mathbb{H}^{q_{1}\times n},F\in \mathbb{H}^{q_{2}\times n}$  and $G\in \mathbb{H}^{q_{3}\times n}$ be given. Assume that  equation (\ref{system002}) is consistent. Then,
\begin{align*}
\mathop {\max }\limits_{BXE+CYF+DZG=A } r \left({X }\right)
 =\min\left\{ p_{1},~q_{1},~p_{1}+q_{1}+r_{a|f|g}-r_{b}-r_{e|f|g},~p_{1}+q_{1}+r_{acd}-r_{e}-r_{bcd},\right.
\end{align*}
\begin{align*}
p_{1}+q_{1}+r_{b0ad0|ba00c|0g000|00f00}-r_{gg|e0|0f}-r_{bd}-r_{bc},~p_{1}+q_{1}+r_{ac|g0}-r_{bc}-r_{f|g},
\end{align*}
\begin{align*}
\left. p_{1}+q_{1}+r_{ad|f0}-r_{e|f}-r_{bd},~p_{1}+q_{1}+r_{ee00|0ad0|a00c|g000|0f00}-r_{db0|d0c}-r_{e|g}-r_{e|f}\right\},
\end{align*}
\begin{align*}
&\mathop {\min }\limits_{ BXE+CYF+DZG=A } r \left({X }\right)
\\=&r_{b0ad0|ba00c|0g000|00f00}+r_{ee00|0ad0|a00c|g000|0f00}+r_{acd}+r_{a|f|g}+r_{bd}+r_{e|g}+r_{cd}+r_{f|g}\\&
-r_{b0ad00|ba00cd|0g0000|00f000}-r_{ee00|0ad0|a00c|g000|0f00|0g00}
\\&+ \max\big\{ r_{ac|g0}-r_{b0ad0|ba00c|0g000|00f00|00g00}-r_{ee000|0ad0c|a00c0|g0000|0f000},\\&\qquad \qquad
r_{ad|f0}-r_{b0ad0|ba00c|0g000|00f00|0f000}-r_{ee000|0ad00|a00cd|g0000|0f000}\big\}.
\end{align*}
\end{theorem}

\begin{theorem}
Let $A\in \mathbb{H}^{m\times n},
B\in \mathbb{H}^{m\times p_{1}},C\in \mathbb{H}^{m\times p_{2}},D\in \mathbb{H}^{m\times p_{3}},E\in \mathbb{H}^{q_{1}\times n},F\in \mathbb{H}^{q_{2}\times n}$  and $G\in \mathbb{H}^{q_{3}\times n}$ be given. Assume that  equation (\ref{system002}) is consistent.
Then,
\begin{align*}
\mathop {\max }\limits_{BXE+CYF+DZG=A } r \left({Y }\right)
 =\min\left\{ p_{2},~q_{2},~p_{2}+q_{2}+r_{a|e|g}-r_{c}-r_{e|f|g},~p_{2}+q_{2}+r_{abd}-r_{f}-r_{bcd},\right.
\end{align*}
\begin{align*}
p_{2}+q_{2}+r_{c0ab0|ca00d|0e000|00g00}-r_{gg|e0|0f}-r_{bc}-r_{cd},~p_{2}+q_{2}+r_{ad|e0}-r_{cd}-r_{e|g},
\end{align*}
\begin{align*}
\left. p_{2}+q_{2}+r_{ab|g0}-r_{f|g}-r_{bc},~p_{2}+q_{2}+r_{ff00|0ab0|a00d|e000|0g00}-r_{db0|d0c}-r_{e|f}-r_{f|g}\right\},
\end{align*}
\begin{align*}
&\mathop {\min }\limits_{ BXE+CYF+DZG=A } r \left({Y }\right)
\\=&r_{c0ab0|ca00d|0e000|00g00}+r_{ff00|0ab0|a00d|e000|0g00}+r_{abd}+r_{a|e|g}+r_{bc}+r_{e|f}+r_{bd}+r_{e|g}\\&
-r_{c0ab00|ca00db|0e0000|00g000}-r_{ff00|0ab0|a00d|e000|0g00|0e00}
\\&+ \max\big\{ r_{ad|e0}-r_{c0ab0|ca00d|0e000|00g00|00e00}-r_{ff000|0ab0d|a00d0|e0000|0g000},\\&\qquad \qquad
r_{ab|g0}-r_{c0ab0|ca00d|0e000|00g00|0g000}-r_{ff000|0ab00|a00db|e0000|0g000}\big\}.
\end{align*}
\end{theorem}

\subsection{\textbf{Some solvability conditions and the   general
solution to  (\ref{system001})}}
In this section, the simultaneous decomposition of  (\ref{array1}) will be used to   solve the real quaternion matrix
equation (\ref{system001}).

\begin{theorem}\label{theorem05}
Let $A\in \mathbb{H}^{m\times n},
B\in \mathbb{H}^{m\times p_{1}},C\in \mathbb{H}^{m\times p_{2}},D\in \mathbb{H}^{m\times p_{3}},E\in \mathbb{H}^{q_{1}\times n},F\in \mathbb{H}^{q_{2}\times n}$  and $G\in \mathbb{H}^{q_{3}\times n}$ be given.
Then  the  equation (\ref{system001}) is consistent if and only if
\begin{align}\label{solvab340}
r_{abcd|e000|f000|g000}=r_{bcd}+r_{e|f|g},~  A_{89}=0,~A_{98}=0,~A_{i,10}=0,~A_{10,i}=0,~ (i=6,7,8,9).
\end{align}

In this case, the general solution to (\ref{system001}) can be expressed as
\begin{align*}
X=T_{1}^{-1}\widehat{X}Q, \quad W=P\widehat{W}V_{1}^{-1},\quad Y=T_{2}^{-1}\widehat{Y}V_{2}^{-1},\quad Z=T_{3}^{-1}\widehat{Z}V_{3}^{-1},
\end{align*}
where
\begin{align}\label{equh340}
\widehat{X}=\bordermatrix{
~& n_{1}&n_{2} & n_{3}&n_{4}&n_{5}&n_{4}&n_{6}&n_{7}&n_{8}&n-r_{e|f|g} \cr
m_{1}&X_{11}&X_{12}&X_{13}&X_{14}&X_{15}&X_{16}&X_{17}&X_{18}&X_{19}&A_{1,10}  \cr
m_{2}&X_{21}&X_{22}&X_{23}&X_{24}&X_{25}&X_{26}&X_{27}&X_{28}&A_{29}&A_{2,10}  \cr
m_{3}&X_{31}&X_{32}&X_{33}&X_{34}&X_{35}&X_{36}&X_{37}&A_{38}&X_{39}&A_{3,10}  \cr
m_{4}&X_{41}&X_{42}&X_{43}&X_{44}&X_{45}&X_{46}&X_{47}&A_{48}&A_{49}-A_{69}&A_{4,10}  \cr
m_{5}&X_{51}&X_{52}&X_{53}&X_{54}&X_{55}&A_{56}&A_{57}&A_{58}&A_{59}&A_{5,10}  \cr
p_{1}-r_{b}&X_{61}&X_{62}&X_{63}&X_{64}&X_{65}&X_{66}&X_{67}&X_{68}&X_{69}&X_{6,10} },
\end{align}
\begin{align}
\widehat{W}=\bordermatrix{
~&\scriptstyle   n_{1}&\scriptstyle  n_{2} &\scriptstyle   n_{3}&\scriptstyle  n_{4}&\scriptstyle  n_{5}&\scriptstyle   q_{2}-r_{e} \cr
\scriptstyle m_{1}&\scriptstyle  W_{11} &\scriptstyle  W_{12}&\scriptstyle  W_{13}&\scriptstyle  W_{14}&\scriptstyle  A_{15}-X_{15}&\scriptstyle  W_{16} \cr
\scriptstyle m_{2}&\scriptstyle  W_{21} &\scriptstyle  W_{22} &\scriptstyle  A_{23}-X_{23}&\scriptstyle  A_{24}-X_{24}&\scriptstyle  A_{25}-X_{25}&\scriptstyle  W_{26}  \cr
\scriptstyle m_{3}&\scriptstyle  W_{31} &\scriptstyle  A_{32}-X_{32}&\scriptstyle  W_{33} &\scriptstyle  W_{34} &\scriptstyle  A_{35}-X_{35}&\scriptstyle  W_{36}  \cr
\scriptstyle m_{4}&\scriptstyle  W_{41} &\scriptstyle  A_{42}-X_{42}&\scriptstyle  A_{43}-A_{63}+W_{63}-X_{43} &\scriptstyle  A_{44}-A_{46}+X_{46}-X_{44} &\scriptstyle  A_{45}-X_{45}&\scriptstyle  W_{46} \cr
\scriptstyle m_{5}&\scriptstyle  A_{51}-X_{51}&\scriptstyle  A_{52}-X_{52}&\scriptstyle  A_{53}-X_{53}&\scriptstyle  A_{54}-X_{54}&\scriptstyle  A_{55}-X_{55}&\scriptstyle  W_{56} \cr
\scriptstyle m_{4}&\scriptstyle  W_{61}&\scriptstyle  W_{62}&\scriptstyle  W_{63}&\scriptstyle  A_{64}-A_{46}+X_{46}&\scriptstyle  A_{65}&\scriptstyle  W_{66}  \cr
\scriptstyle m_{6}&\scriptstyle  W_{71}&\scriptstyle  W_{72}&\scriptstyle  W_{73}&\scriptstyle  W_{74}&\scriptstyle  A_{75}&\scriptstyle  W_{76} \cr
\scriptstyle m_{7}&\scriptstyle  W_{81}&\scriptstyle  W_{82}&\scriptstyle  A_{83}&\scriptstyle  A_{84}&\scriptstyle  A_{85}&\scriptstyle  W_{86} \cr
\scriptstyle m_{8}&\scriptstyle  W_{91}&\scriptstyle  A_{92}&\scriptstyle  W_{93}&\scriptstyle  A_{94}-A_{96}&\scriptstyle  A_{95}&\scriptstyle  W_{96}  \cr
\scriptstyle m-r_{bcd}&\scriptstyle  A_{10,1}&\scriptstyle  A_{10,2}&\scriptstyle  A_{10,3}&\scriptstyle  A_{10,4}&\scriptstyle  A_{10,5}&\scriptstyle  W_{10,6} },
\end{align}
\begin{align}
&\widehat{Y}=\nonumber\\&\bordermatrix{
~&\scriptscriptstyle  n_{4}&\scriptscriptstyle n_{6} &\scriptscriptstyle  n_{7}&\scriptscriptstyle n_{1}&\scriptscriptstyle n_{2}&\scriptscriptstyle q_{2}-r_{f} \cr
\scriptscriptstyle m_{4}&\scriptscriptstyle  A_{66}-A_{64}+X_{46}&\scriptscriptstyle  A_{67}-A_{47}+X_{47}&\scriptscriptstyle  A_{68}&\scriptscriptstyle  A_{61}-A_{41}+X_{41}+W_{41}-W_{61}&\scriptscriptstyle  A_{62}-W_{62}&\scriptscriptstyle  Y_{16} \cr
\scriptscriptstyle m_{6}&\scriptscriptstyle  A_{76}-A_{74}+W_{74}&\scriptscriptstyle  Y_{22}&\scriptscriptstyle  A_{78}&\scriptscriptstyle  Y_{24}&\scriptscriptstyle  A_{72}-W_{72}&\scriptscriptstyle  Y_{26} \cr
\scriptscriptstyle m_{7}&\scriptscriptstyle  A_{86}&\scriptscriptstyle  A_{87}&\scriptscriptstyle  A_{88}&\scriptscriptstyle  A_{81}-W_{81}&\scriptscriptstyle  A_{82}-W_{82}&\scriptscriptstyle  Y_{36} \cr
\scriptscriptstyle m_{1}&\scriptscriptstyle  A_{16}-A_{14}+X_{14}-X_{16}-W_{14}&\scriptscriptstyle  Y_{42}&\scriptscriptstyle  A_{18}-X_{18}&\scriptscriptstyle  Y_{44}&\scriptscriptstyle  A_{12}-(X_{12}+W_{12})&\scriptscriptstyle  Y_{46} \cr
\scriptscriptstyle m_{2}&\scriptscriptstyle  A_{26}-X_{26}&\scriptscriptstyle  A_{27}-X_{27}&\scriptscriptstyle  A_{28}-X_{28}&\scriptscriptstyle  A_{21}-(X_{21}+W_{21})&\scriptscriptstyle  A_{22}-(X_{22}+W_{22})&\scriptscriptstyle  Y_{56} \cr
\scriptscriptstyle p_{2}-r_{c}&\scriptscriptstyle  Y_{61}&\scriptscriptstyle  Y_{62}&\scriptscriptstyle  Y_{63}&\scriptscriptstyle  Y_{64}&\scriptscriptstyle  Y_{65}&\scriptscriptstyle  Y_{66}  },
\end{align}
\begin{align}\label{equh343}
&\widehat{Z}=\nonumber\\&\bordermatrix{
~&\scriptstyle  n_{8}&\scriptstyle n_{4} &\scriptstyle  n_{6}&\scriptstyle n_{3}&\scriptstyle n_{1}&\scriptstyle q_{3}-r_{g} \cr
\scriptstyle m_{8}&\scriptstyle  A_{99}&\scriptstyle  A_{96}&\scriptstyle  A_{97}&\scriptstyle  A_{93}-W_{93}&\scriptstyle  A_{91}-W_{91}&\scriptstyle  Z_{16} \cr
\scriptstyle m_{4}&\scriptstyle  A_{69}&\scriptstyle  A_{46}-X_{46}&\scriptstyle  A_{47}-X_{47}&\scriptstyle  A_{63}-W_{63}&\scriptstyle  A_{41}-(X_{41}+W_{41})&\scriptstyle  Z_{26} \cr
\scriptstyle m_{6}&\scriptstyle  A_{79}&\scriptstyle  A_{74}-W_{74}&\scriptstyle  A_{77}-Y_{22}&\scriptstyle  A_{73}-W_{73}&\scriptstyle  A_{71}-(Y_{24}+W_{71})&\scriptstyle  Z_{36} \cr
\scriptstyle m_{3}&\scriptstyle  A_{39}-X_{39}&\scriptstyle  A_{34}-(X_{34}+W_{34})&\scriptstyle  A_{37}-X_{37}&\scriptstyle  A_{33}-(X_{33}+W_{33})&\scriptstyle  A_{31}-(X_{31}+W_{31})&\scriptstyle  Z_{46} \cr
\scriptstyle m_{1}&\scriptstyle  A_{19}-X_{19}&\scriptstyle  A_{14}-(X_{14}+W_{14})&\scriptstyle  A_{17}-(X_{17}+Y_{42})&\scriptstyle  A_{13}-(X_{13}+W_{13})&\scriptstyle  Z_{55}&\scriptstyle  Z_{56} \cr
\scriptstyle p_{3}-r_{d}&\scriptstyle  Z_{61}&\scriptstyle  Z_{62}&\scriptstyle  Z_{63}&\scriptstyle  Z_{64}&\scriptstyle  Z_{65}&\scriptstyle  Z_{66} },
\end{align}
$P,Q,A_{ij},T_{i},V_{i}$ are defined in Theorem \ref{theorem01},    the remaining $X_{ij},Y_{ij},Z_{ij}$ in (\ref{equh340})-(\ref{equh343}) are arbitrary matrices over $\mathbb{H}$
with appropriate sizes.
\end{theorem}

\begin{proof}
From Theorem \ref{theorem01}, we know that  rewrite  matrix equation (\ref{system001})  is consistent is equivalent to that the following matrix equation
\begin{align*}
P(S_{B}T_{1}XQ^{-1}+P^{-1}WV_{1}S_{E}+S_{C}T_{2}YV_{2}S_{F}+S_{D}T_{3}ZV_{3}S_{G})Q=PS_{A}Q.
\end{align*}
Because $P,Q$ are nonsingular, the matrix equation (\ref{system001})  can be written as
\begin{align}\label{equ70036}
S_{B}T_{1}XQ^{-1}+P^{-1}WV_{1}S_{E}+S_{C}T_{2}YV_{2}S_{F}+S_{D}T_{3}ZV_{3}S_{G}=S_{A}.
\end{align}Let the matrices
\begin{align} \label{equ70037}
\widehat{X}=T_{1}XQ^{-1}=\begin{pmatrix}X_{11}&\cdots&X_{1,11}\\
\vdots&\ddots&\vdots\\
X_{61}&\cdots&X_{6,11}\end{pmatrix},~\widehat{W}=P^{-1}WV_{1}=\begin{pmatrix}W_{11}&\cdots&W_{1,6}\\
\vdots&\ddots&\vdots\\
W_{11,1}&\cdots&W_{11,6}\end{pmatrix},
\end{align}
\begin{align}\label{equ70038}
\widehat{Y}=T_{2}YV_{2}=\begin{pmatrix}Y_{11}&\cdots&Y_{16}\\
\vdots&\ddots&\vdots\\
Y_{61}&\cdots&Y_{66}\end{pmatrix},~\widehat{Z}=T_{3}ZV_{3}=\begin{pmatrix}Z_{11}&\cdots&Z_{16}\\
\vdots&\ddots&\vdots\\
Z_{61}&\cdots&Z_{66}\end{pmatrix},
\end{align}be partitioned in accordance with (\ref{equ70036}). Substituting (\ref{equ70037}) and (\ref{equ70038}) into (\ref{equ70036}) yields
\begin{align}\label{equh347}
S_{A}-S_{B}\widehat{X}-\widehat{W}S_{E}-S_{C}\widehat{Y}S_{F}-S_{D}\widehat{Z}S_{G}\triangleq\begin{pmatrix}\Omega_{11}&\Omega_{12}\\ \Omega_{21}&\Omega_{22}\end{pmatrix}=0,
\end{align}
where
\begin{align}\label{equh349}
&\Omega_{11}=\nonumber\\&\begin{pmatrix}
\scriptscriptstyle A_{11}-(X_{11}+W_{11}+Y_{44}+Z_{55})&\scriptscriptstyle   \scriptscriptstyle   A_{12}-(X_{12}+W_{12}+Y_{45})&\scriptscriptstyle   \scriptscriptstyle   A_{13}-(X_{13}+W_{13}+Z_{54})&\scriptscriptstyle   \scriptscriptstyle   A_{14}-(X_{14}+W_{14}+Z_{52})&\scriptscriptstyle   \scriptscriptstyle   A_{15}-(X_{15}+W_{15})\\
\scriptscriptstyle A_{21}-(X_{21}+W_{21}+Y_{54})&\scriptscriptstyle   A_{22}-(X_{22}+W_{22}+Y_{55})&\scriptscriptstyle   A_{23}-(X_{23}+W_{23})&\scriptscriptstyle   A_{24}-(X_{24}+W_{24})&\scriptscriptstyle   A_{25}-(X_{25}+W_{25})\\
\scriptscriptstyle A_{31}-(X_{31}+W_{31}+Z_{45})&\scriptscriptstyle   A_{32}-(X_{32}+W_{32})&\scriptscriptstyle   A_{33}-(X_{33}+W_{33}+Z_{44})&\scriptscriptstyle   A_{34}-(X_{34}+W_{34}+Z_{42})&\scriptscriptstyle   A_{35}-(X_{35}+W_{35})\\
\scriptscriptstyle A_{41}-(X_{41}+W_{41}+Z_{25})&\scriptscriptstyle   A_{42}-(X_{42}+W_{42})&\scriptscriptstyle   A_{43}-(X_{43}+W_{43}+Z_{24})&\scriptscriptstyle   A_{44}-(X_{44}+W_{44}+Z_{22})&\scriptscriptstyle   A_{45}-(X_{45}+W_{45})\\
\scriptscriptstyle A_{51}-(X_{51}+W_{51})&\scriptscriptstyle   A_{52}-(X_{52}+W_{52})&\scriptscriptstyle   A_{53}-(X_{53}+W_{53})&\scriptscriptstyle   A_{54}-(X_{54}+W_{54})&\scriptscriptstyle   A_{55}-(X_{55}+W_{55})
\end{pmatrix},
\end{align}
\begin{align}
&\Omega_{12}=\nonumber\\&\begin{pmatrix}
\scriptstyle  A_{16}-(Y_{41}+Z_{52}+X_{16}), &\scriptstyle   A_{17}-(Y_{42}+Z_{53}+X_{17}), &\scriptstyle   A_{18}-(Y_{43}+X_{18}), &\scriptstyle   A_{19}-(Z_{51}+X_{19}), &\scriptstyle   A_{1,10}-X_{1,10},  &\scriptstyle    -X_{1,11}  \\
\scriptstyle  A_{26}-(Y_{51}+X_{26}), &\scriptstyle   A_{27}-(Y_{52}+X_{27}),&\scriptstyle    A_{28}-(Y_{53}+X_{28}), &\scriptstyle    A_{29}-X_{29}, &\scriptstyle    A_{2,10}- X_{2,10}, &\scriptstyle    -X_{2,11} \\
\scriptstyle  A_{36}-(Z_{42}+X_{36}), &\scriptstyle   A_{37}-(Z_{43}+X_{37}),&\scriptstyle   A_{38}-X_{38}, &\scriptstyle   A_{39}-(Z_{41}+X_{39)}, &\scriptstyle    A_{3,10}- X_{3,10},&\scriptstyle    -X_{3,11} \\
\scriptstyle  A_{46}-(Z_{22}+X_{46}), &\scriptstyle   A_{47}-(Z_{23}+X_{47}), &\scriptstyle   A_{48}-X_{48},&\scriptstyle   A_{49}-(Z_{21}+X_{49}), &\scriptstyle   A_{4,10}-X_{4,10}, &\scriptstyle    -X_{4,11} \\
\scriptstyle  A_{56}-X_{56}, &\scriptstyle   A_{57}-X_{57} &\scriptstyle   A_{58}- X_{58}, &\scriptstyle   A_{59}- X_{59},&\scriptstyle    A_{5,10}-X_{5,10},&\scriptstyle    -X_{5,11}
\end{pmatrix},
\end{align}
\begin{align}
\Omega_{21}=\begin{pmatrix}
\scriptstyle   A_{61}-(Y_{14}+W_{61}+Z_{25}) &\scriptstyle   A_{62}-(Y_{15}+W_{62}) &\scriptstyle   A_{63}-(Z_{24}+W_{63}) &\scriptstyle   A_{64}-(Z_{22}+W_{64}) &\scriptstyle    A_{65}-W_{65}\\
\scriptstyle   A_{71}-(Y_{24}+W_{71} +Z_{35})&\scriptstyle   A_{72}-(Y_{25}+W_{72}) &\scriptstyle   A_{73}-(Z_{34}+W_{73}) &\scriptstyle   A_{74}-(Z_{32}+W_{74}) &\scriptstyle    A_{75}-W_{75} \\
\scriptstyle   A_{81}-(Y_{34} +W_{81})&\scriptstyle   A_{82}-(Y_{35}+W_{82})&\scriptstyle   A_{83}-W_{83} &\scriptstyle   A_{84}-W_{84} &\scriptstyle    A_{85}-W_{85}\\
\scriptstyle   A_{91}-(Z_{15}+W_{91}) &\scriptstyle    A_{92}-W_{92}&\scriptstyle   A_{93}-(Z_{14}+W_{93}) &\scriptstyle   A_{94}-(Z_{12}+W_{94}) &\scriptstyle   A_{95}- W_{95}\\
\scriptstyle   A_{10,1}-W_{10,1}&\scriptstyle    A_{10,2}-W_{10,2} &\scriptstyle   A_{10,3}-W_{10,3} &\scriptstyle   A_{10,4}-W_{10,4}&\scriptstyle   A_{10,5}-W_{10,5} \\
-W_{11,1} &\scriptstyle    -W_{11,2}&\scriptstyle   -W_{11,3} &\scriptstyle   - W_{11,4} &\scriptstyle    -  W_{11,5}
\end{pmatrix},
\end{align}
\begin{align}\label{equh352}
\Omega_{22}=\begin{pmatrix}
\scriptstyle  A_{66}-(Y_{11} +Z_{22})&\scriptstyle   A_{67}-(Y_{12} +Z_{23}) &\scriptstyle   A_{68}-Y_{13}&\scriptstyle   A_{69}-Z_{21} &\scriptstyle      A_{6,10} &\scriptstyle    0 \\
\scriptstyle A_{76}-(Y_{21} +Z_{32})&\scriptstyle   A_{77}-(Y_{22}+Z_{33}) &\scriptstyle   A_{78}-Y_{23} &\scriptstyle    A_{79}-Z_{31}&\scriptstyle     A_{7,10} &\scriptstyle    0 \\
\scriptstyle A_{86}-Y_{31} &\scriptstyle   A_{87}-Y_{32}  &\scriptstyle   A_{88}-Y_{33} &\scriptstyle    A_{89} &\scriptstyle     A_{8,10}&\scriptstyle    0\\
\scriptstyle A_{96}-Z_{12} &\scriptstyle   A_{97}-Z_{13} &\scriptstyle    A_{98} &\scriptstyle    A_{99}-Z_{11} &\scriptstyle      A_{9,10} &\scriptstyle    0 \\
\scriptstyle  A_{10,6} &\scriptstyle   A_{10,7} &\scriptstyle    A_{10,8}&\scriptstyle   A_{10,9} &\scriptstyle      0&\scriptstyle   0\\
0 &\scriptstyle   0&\scriptstyle    0 &\scriptstyle    0 &\scriptstyle    0 &\scriptstyle     I_{t}
\end{pmatrix}.
\end{align}

If the equation (\ref{system001}) has a solution $(X,W,Y,Z)$, then by (\ref{equh347}), we have that the equalities in
(\ref{solvab340}) hold, and
$$\Omega_{11}=0,\Omega_{12}=0,\Omega_{21}=0,\Omega_{22}=0.$$

Conversely, assume that the equalities in (\ref{solvab340}) hold, then by (\ref{equ0022})-(\ref{equ0024}) and  (\ref{equh349})-(\ref{equh352}), it can be
verified that the matrices have the forms of (\ref{equh340})-(\ref{equh343}) is a solution of (\ref{equ70036}), i.e., (\ref{system002}).

\end{proof}

\begin{remark}
The presented expressions of $X$ and $W$ are more useful   than the expressions found by Wang and He \cite{wanghe}, since the latter can not be used to discuss the range of ranks of $X$ and $W$ to (\ref{system001}).
\end{remark}

\subsection{\textbf{The range of ranks  of the   general
solution to   (\ref{system001})}}
We in this section
consider the range of ranks  of the   general
solution to   (\ref{system001}).

\begin{theorem}Let $A\in \mathbb{H}^{m\times n},
B\in \mathbb{H}^{m\times p_{1}},C\in \mathbb{H}^{m\times p_{2}},D\in \mathbb{H}^{m\times p_{3}},E\in \mathbb{H}^{q_{1}\times n},F\in \mathbb{H}^{q_{2}\times n}$  and $G\in \mathbb{H}^{q_{3}\times n}$ be given. Assume that  equation (\ref{system001}) is consistent.
Then,
\begin{align*}
\mathop {\max }\limits_{BX+WE+CYF+DZG=A } r \left({X }\right)
 =&\min\big\{  p_{1},~n,~p_{1}+r_{acd|e00}-r_{bcd},~p_{1}+r_{a|e|f|g}-r_{b},\\&
 p_{1}+r_{b0ad0|ba00c|0g000|00f00|0e000|00e00}-r_{e|f|g}-r_{bc}-r_{bd},\\&\qquad p_{1}+r_{ac|e0|g0}-r_{bc},~p_{1}+r_{ad|e0|f0}-r_{bd}\big\},
\end{align*}
\begin{align*}
&\mathop {\min }\limits_{ BX+WE+CYF+DZG=A } r \left({X }\right)
\\=&r_{a|e|f|g}+r_{acd|e00}-r_{e}-r_{b0ad00|ba00cd|0g0000|00f000|0e0000|00e000}\\&+r_{bd}+r_{cd}+r_{b0ad0|ba00c|0g000|00f00|0e000|00e00}+r_{e|f|g}\\&+
\min\big\{r_{ac|e0|g0}-r_{b0ad0|ba00c|0g000|00f00|00g00|0e000|00e00}-r_{ad0c|a0c0|e000|f000|g000},
\\&\qquad \qquad r_{ad|e0|f0}-r_{b0ad0|ba00c|0g000|00f00|0e000|00e00|0f000}-r_{ad00|a0cd|e000|f000|g000} \big\}.
\end{align*}
\end{theorem}

\begin{proof}
It follows from  Theorem \ref{theorem05} that the expression of $X$ in (\ref{system001}) can be expressed as $X=T_{1}^{-1}\widehat{X}Q,$ where $\widehat{X}$ is given in (\ref{equh340}). Clearly, $r(X)=r(T_{1}^{-1}\widehat{X}Q)=r(\widehat{X}).$ Now we consider the
maximal and minimal ranks of $\widehat{X}$. Applying Lemma \ref{lemma02} to the variable matrices $\begin{pmatrix}\begin{smallmatrix}X_{11}&\cdots&X_{15}\\ \vdots&\ddots&\vdots\\X_{51}&\cdots&X_{55}\end{smallmatrix}\end{pmatrix}$ and $(X_{66},~X_{67},~X_{68},~X_{69},~X_{6,10})$ of $\widehat{X}$, we obtain
\begin{align*}
&\mathop {\max }\limits_{ \begin{pmatrix}\begin{smallmatrix}X_{11}&\cdots&X_{15}\\ \vdots&\ddots&\vdots\\X_{51}&\cdots&X_{55}\end{smallmatrix}\end{pmatrix},
 X_{6i},(i=6,\ldots,10)} r ({\widehat{X }})
 \\&=\min\left\{ p_{1},n,p_{1}-r_{b}+r_{e}+r(\Phi_{1}),r(X_{61},X_{62},X_{63},X_{64},X_{65})+r_{b}+n_{4}+n_{6}+n_{7}+n_{8}+n-r_{e|f|g}\right\},
\end{align*}
\begin{align*}
\mathop {\min }\limits_{ \begin{pmatrix}\begin{smallmatrix}X_{11}&\cdots&X_{15}\\ \vdots&\ddots&\vdots\\X_{51}&\cdots&X_{55}\end{smallmatrix}\end{pmatrix},
 X_{6i},(i=6,\ldots,10)} r ({\widehat{X }})
 = \mathop {\max }\left\{r(\Phi_{1}),~r(X_{61},X_{62},X_{63},X_{64},X_{65})\right\},
\end{align*}
where
\begin{align*}
\Phi_{1}=\bordermatrix{
~&  n_{4}&n_{6}&n_{7}&n_{8}&n-r_{e|f|g} \cr
m_{1}& X_{16}&X_{17}&X_{18}&X_{19}&A_{1,10}  \cr
m_{2}& X_{26}&X_{27}&X_{28}&A_{29}&A_{2,10}  \cr
m_{3}& X_{36}&X_{37}&A_{38}&X_{39}&A_{3,10}  \cr
m_{4}& X_{46}&X_{47}&A_{48}&A_{49}-A_{69}&A_{4,10}  \cr
m_{5}& A_{56}&A_{57}&A_{58}&A_{59}&A_{5,10}    }.
\end{align*}
Note that
$$\mathop {\max }\left\{ r(X_{61},X_{62},X_{63},X_{64},X_{65})\right\}=\mathop {\min }\left\{p_{1}-r_{b},~r_{b}\right\},
~\mathop {\min }\left\{ r(X_{61},X_{62},X_{63},X_{64},X_{65})\right\}=0.$$
Hence, we obtain
\begin{align*}
\mathop {\max }\limits_{ \begin{pmatrix}\begin{smallmatrix}X_{11}&\cdots&X_{15}\\ \vdots&\ddots&\vdots\\X_{51}&\cdots&X_{55}\end{smallmatrix}\end{pmatrix},
 X_{6,i},(i=1,\ldots,10)} r ({\widehat{X }})
 =\min\left\{ p_{1},n,p_{1}-r_{b}+r_{e}+r(\Phi_{1})\right\},
\end{align*}
\begin{align*}
\mathop {\min }\limits_{ \begin{pmatrix}\begin{smallmatrix}X_{11}&\cdots&X_{15}\\ \vdots&\ddots&\vdots\\X_{51}&\cdots&X_{55}\end{smallmatrix}\end{pmatrix},
 X_{6,i},(i=1,\ldots,10)} r ({\widehat{X }})
 = r(\Phi_{1}).
\end{align*}
Applying Lemma \ref{lemma03} to the variable matrices $\begin{pmatrix}\begin{smallmatrix}X_{16}&X_{17}\\X_{26}&X_{27}\\X_{36}&X_{37}\\X_{46}&X_{47} \end{smallmatrix}\end{pmatrix}$ of $\Phi_{1}$, we obtain
\begin{align*}
&\mathop {\max }\limits_{ \begin{pmatrix}\begin{smallmatrix}X_{16}&X_{17}\\X_{26}&X_{27}\\X_{36}&X_{37}\\X_{46}&X_{47} \end{smallmatrix}\end{pmatrix}} r \left({\Phi_{1} }\right)
 =\nonumber\\&\min\left\{ m_{1}+m_{2}+m_{3}+m_{4}+r\begin{pmatrix}A_{56},&A_{57},&A_{58},&A_{59},&A_{5,10}\end{pmatrix},n_{4}+n_{6}+r(\Phi_{2})\right\},
\end{align*}
\begin{align*}
\mathop {\min }\limits_{ \begin{pmatrix}\begin{smallmatrix}X_{16}&X_{17}\\X_{26}&X_{27}\\X_{36}&X_{37}\\X_{46}&X_{47} \end{smallmatrix}\end{pmatrix}} r \left({\Phi_{1} }\right)
 = r(\Phi_{2})+r\begin{pmatrix}A_{56},&A_{57},&A_{58},&A_{59},&A_{5,10}\end{pmatrix}-r\begin{pmatrix}A_{58},&A_{59},&A_{5,10}\end{pmatrix},
\end{align*}
where
\begin{align*}
\Phi_{2}=\bordermatrix{
~&   n_{7}&n_{8}&n-r_{e|f|g} \cr
m_{1}&  X_{18}&X_{19}&A_{1,10}  \cr
m_{2}&  X_{28}&A_{29}&A_{2,10}  \cr
m_{3}& A_{38}&X_{39}&A_{3,10}  \cr
m_{4}&  A_{48}&A_{49}-A_{69}&A_{4,10}  \cr
m_{5}&  A_{58}&A_{59}&A_{5,10}    }.
\end{align*}
Applying Lemma \ref{lemma03} to the variable matrices $(X_{18},~X_{19})$ of $\Phi_{2}$, we obtain
\begin{align*}
\mathop {\max }\limits_{ (X_{18},~X_{19}) } r \left({\Phi_{2} }\right)
 =\min\left\{ n_{7}+n_{8}+r\begin{pmatrix}A_{1,10}\\A_{2,10}\\A_{3,10}\\A_{4,10}\\A_{5,10}\end{pmatrix},m_{1}+r(\Phi_{3})\right\},
\end{align*}
\begin{align*}
\mathop {\min }\limits_{ (X_{18},~X_{19}) } r \left({\Phi_{2} }\right)
 = r(\Phi_{3})+r\begin{pmatrix}A_{1,10}\\A_{2,10}\\A_{3,10}\\A_{4,10}\\A_{5,10}\end{pmatrix}
 -r\begin{pmatrix}A_{2,10}\\A_{3,10}\\A_{4,10}\\A_{5,10}\end{pmatrix},
\end{align*}
where
\begin{align*}
\Phi_{3}=\bordermatrix{
~&   n_{7}&n_{8}&n-r_{e|f|g} \cr
m_{2}&  X_{28}&A_{29}&A_{2,10}  \cr
m_{3}& A_{38}&X_{39}&A_{3,10}  \cr
m_{4}&  A_{48}&A_{49}-A_{69}&A_{4,10}  \cr
m_{5}&  A_{58}&A_{59}&A_{5,10}    }.
\end{align*}
Applying Lemma \ref{lemma04} to the variable matrices $X_{28}$ and $X_{39}$ of $\Phi_{3}$, we obtain
\begin{align*}
&\mathop {\max }\limits_{ X_{28},X_{39}} r \left({\Phi_{3}}\right)
 \nonumber\\=&\min\left\{ m_{2}+m_{3}+r\begin{pmatrix}A_{48}&A_{49}-A_{69}&A_{4,10}\\
A_{58}&A_{59}&A_{5,10}\end{pmatrix},
 m_{2}+n_{8}+r\begin{pmatrix}A_{38}&A_{3,10}\\A_{48}&A_{4,10}\\A_{58}&A_{5,10}\end{pmatrix},\right.
\end{align*}
\begin{align*}
\left.  m_{3}+n_{7}+r\begin{pmatrix}A_{29}&A_{2,10}\\A_{49}-A_{69}&A_{4,10}\\A_{59}&A_{5,10}\end{pmatrix},
n_{7}+n_{8}+r\begin{pmatrix}A_{2,10}\\A_{3,10}\\A_{4,10}\\A_{5,10}\end{pmatrix}
\right\},
\end{align*}
\begin{align*}
\mathop {\min }\limits_{ X_{28},X_{39} } r \left({\Phi_{3}}\right)
  =&r\begin{pmatrix}A_{48}&A_{49}-A_{69}&A_{4,10}\\
A_{58}&A_{59}&A_{5,10}\end{pmatrix}+r\begin{pmatrix}A_{2,10}\\A_{3,10}\\A_{4,10}\\A_{5,10}\end{pmatrix}\\&+\max\left\{ r\begin{pmatrix}A_{38}&A_{3,10}\\A_{48}&A_{4,10}\\A_{58}&A_{5,10}\end{pmatrix}
  -r\begin{pmatrix}A_{48}&A_{4,10}\\A_{58}&A_{5,10}\end{pmatrix}
  -r\begin{pmatrix}A_{3,10}\\A_{4,10}\\A_{5,10}\end{pmatrix},\right.
\end{align*}
\begin{align*}
\left.   r\begin{pmatrix}A_{29}&A_{2,10}\\A_{49}-A_{69}&A_{4,10}\\A_{59}&A_{5,10}\end{pmatrix}-r\begin{pmatrix}A_{49}-A_{69}&A_{4,10}\\A_{59}&A_{5,10}\end{pmatrix}-
r\begin{pmatrix}A_{2,10}\\A_{4,10}\\A_{5,10}\end{pmatrix}
\right\}.
\end{align*}
Hence,
\begin{align*}
\mathop {\max }\limits_{BX+WE+CYF+DZG=A } r \left({X }\right)
 =\min\left\{ p_{1},n,t_{1},t_{2},t_{3},t_{4},t_{5},t_{6}\right\},
\end{align*}
\begin{align*}
\mathop {\min }\limits_{BX+WE+CYF+DZG=A } r \left({X }\right)
 =\max\left\{ t_{7},t_{8}\right\},
\end{align*}
where
\begin{align*}
t_{1}=p_{1}-r_{b}+r_{e}+m_{1}+m_{2}+m_{3}+m_{4}+r\begin{pmatrix}A_{56},&A_{57},&A_{58},&A_{59},&A_{5,10}\end{pmatrix},
\end{align*}
\begin{align*}
t_{2}=p_{1}-r_{b}+r_{e}+n_{4}+n_{6}+n_{7}+n_{8}+r\begin{pmatrix}A_{1,10}\\A_{2,10}\\A_{3,10}\\A_{4,10}\\A_{5,10}\end{pmatrix},
\end{align*}
\begin{align*}
t_{3}=p_{1}-r_{b}+r_{e}+n_{4}+n_{6}+m_{1}+m_{2}+m_{3}+r\begin{pmatrix}A_{48}&A_{49}-A_{69}&A_{4,10}\\
A_{58}&A_{59}&A_{5,10}\end{pmatrix},
\end{align*}
\begin{align*}
t_{4}=p_{1}-r_{b}+r_{e}+n_{4}+n_{6}+m_{1}+m_{2}+n_{8}+r\begin{pmatrix}A_{38}&A_{3,10}\\A_{48}&A_{4,10}\\A_{58}&A_{5,10}\end{pmatrix},
\end{align*}
\begin{align*}
t_{5}=p_{1}-r_{b}+r_{e}+n_{4}+n_{6}+m_{1}+m_{3}+n_{7}+r\begin{pmatrix}A_{29}&A_{2,10}\\A_{49}-A_{69}&A_{4,10}\\A_{59}&A_{5,10}\end{pmatrix},
\end{align*}
\begin{align*}
t_{6}=p_{1}-r_{b}+r_{e}+n_{4}+n_{6}+m_{1}+n_{7}+n_{8}+r\begin{pmatrix}A_{2,10}\\A_{3,10}\\A_{4,10}\\A_{5,10}\end{pmatrix}\geq t_{2},
\end{align*}
\begin{align*}
t_{7}=&r\begin{pmatrix}A_{1,10}\\A_{2,10}\\A_{3,10}\\A_{4,10}\\A_{5,10}\end{pmatrix}+r\begin{pmatrix}A_{56},&A_{57},&A_{58},&A_{59},&A_{5,10}\end{pmatrix}
+r\begin{pmatrix}A_{48}&A_{49}-A_{69}&A_{4,10}\\
A_{58}&A_{59}&A_{5,10}\end{pmatrix}
 \\&-r(A_{58},~A_{59},~A_{5,10})+r\begin{pmatrix}A_{38}&A_{3,10}\\A_{48}&A_{4,10}\\A_{58}&A_{5,10}\end{pmatrix}
  -r\begin{pmatrix}A_{48}&A_{4,10}\\A_{58}&A_{5,10}\end{pmatrix}
  -r\begin{pmatrix}A_{3,10}\\A_{4,10}\\A_{5,10}\end{pmatrix},
\end{align*}
\begin{align*}
t_{8}=&r\begin{pmatrix}A_{1,10}\\A_{2,10}\\A_{3,10}\\A_{4,10}\\A_{5,10}\end{pmatrix}+r\begin{pmatrix}A_{56},&A_{57},&A_{58},&A_{59},&A_{5,10}\end{pmatrix}
+r\begin{pmatrix}A_{48}&A_{49}-A_{69}&A_{4,10}\\
A_{58}&A_{59}&A_{5,10}\end{pmatrix}\\&-r(A_{58},~A_{59},~A_{5,10})+r\begin{pmatrix}A_{29}&A_{2,10}\\A_{49}-A_{69}&A_{4,10}\\A_{59}&A_{5,10}\end{pmatrix}-r\begin{pmatrix}A_{49}-A_{69}&A_{4,10}\\A_{59}&A_{5,10}\end{pmatrix}-
r\begin{pmatrix}A_{2,10}\\A_{4,10}\\A_{5,10}\end{pmatrix}.
\end{align*}
Now we pay attention to the ranks of the block matrices in $t_{i}$. Upon construction and computation, we obtain
\begin{align}\label{equh366}
r\begin{pmatrix}A_{56},&A_{57},&A_{58},&A_{59},&A_{5,10}\end{pmatrix}=&r\begin{pmatrix}S_{A}&S_{C}&S_{D}\\S_{E}&0&0\end{pmatrix}-r(S_{C},~S_{D})-r(S_{E})
\nonumber\\=&r_{acd|e00}-r_{cd}-r_{e},
\end{align}
\begin{align}
r\begin{pmatrix}A_{1,10}\\A_{2,10}\\A_{3,10}\\A_{4,10}\\A_{5,10}\end{pmatrix}=r\begin{pmatrix}S_{A}\\S_{E}\\S_{F}\\S_{G}\end{pmatrix}
-r\begin{pmatrix}S_{E}\\S_{F}\\S_{G}\end{pmatrix}
=r_{a|e|f|g}-r_{e|f|g},
\end{align}
\begin{align}
r\begin{pmatrix}A_{48}&A_{49}-A_{69}&A_{4,10}\\
A_{58}&A_{59}&A_{5,10}\end{pmatrix}=
&r\begin{pmatrix}S_{B}&0&S_{A}&S_{D}&0\\S_{B}&S_{A}&0&0&S_{C}\\0&S_{G}&0&0&0\\0&0&S_{F}&0&0\\0&S_{E}&0&0&0\\0&0&S_{E}&0&0\end{pmatrix}
-r(S_{C})-r(S_{D})\nonumber\\&-r\begin{pmatrix}S_{E}\\S_{F}\end{pmatrix}-r\begin{pmatrix}S_{E}\\S_{G}\end{pmatrix}-m_{2}-m_{3}-m_{4}-m_{5}
\nonumber\\
=&r_{b0ad0|ba00c|0g000|00f00|0e000|00e00}\nonumber\\&-r_{c}-r_{d}-r_{e|f}-r_{e|g}-m_{2}-m_{3}-m_{4}-m_{5},
\end{align}
\begin{align}
r\begin{pmatrix}A_{38}&A_{3,10}\\A_{48}&A_{4,10}\\A_{58}&A_{5,10}\end{pmatrix}=r\begin{pmatrix}S_{A}&S_{C}\\S_{E}&0\\S_{G}&0\end{pmatrix}
-r\begin{pmatrix}S_{E}\\S_{G}\end{pmatrix}-r(S_{C})=r_{ac|e0|g0}-r_{e|g}-r_{c},
\end{align}
\begin{align}
r\begin{pmatrix}A_{29}&A_{2,10}\\A_{49}-A_{69}&A_{4,10}\\A_{59}&A_{5,10}\end{pmatrix}=r\begin{pmatrix}S_{A}&S_{D}\\S_{E}&0\\S_{F}&0\end{pmatrix}
-r(S_{D})-r\begin{pmatrix}S_{E}\\S_{F}\end{pmatrix}=r_{ad|e0|f0}-r_{d}-r_{e|f},
\end{align}
\begin{align}
r\begin{pmatrix}A_{3,10}\\A_{4,10}\\A_{5,10}\end{pmatrix}=&r\begin{pmatrix}S_{A}&S_{D}&0&S_{C}\\S_{A}&0&S_{C}&0\\S_{E}&0&0&0\\S_{F}&0&0&0\\S_{G}&0&0&0\end{pmatrix}
-r(S_{C},~S_{D})-r(S_{C})-r\begin{pmatrix}S_{E}\\S_{F}\\S_{G}\end{pmatrix}\nonumber\\
=&r_{ad0c|a0c0|e000|f000|g000}-r_{cd}-r_{c}-r_{e|f|g},
\end{align}
\begin{align}
r\begin{pmatrix}A_{2,10}\\A_{4,10}\\A_{5,10}\end{pmatrix}
=&r\begin{pmatrix}S_{A}&S_{D}&0&0\\S_{A}&0&S_{C}&S_{D}\\S_{E}&0&0&0\\S_{F}&0&0&0\\S_{G}&0&0&0\end{pmatrix}-r(S_{C},~S_{D})-r(S_{D})
-r\begin{pmatrix}S_{E}\\S_{F}\\S_{G}\end{pmatrix}\nonumber\\
=&r_{ad00|a0cd|e000|f000|g000}-r_{cd}-r_{d}-r_{e|f|g},
\end{align}
\begin{align}
r(A_{58},~A_{59},~A_{5,10})=
&r\begin{pmatrix}S_{B}&0&S_{A}&S_{D}&0&0\\S_{B}&S_{A}&0&0&S_{C}&S_{D}\\0&S_{G}&0&0&0&0\\0&0&S_{F}&0&0&0\\0&S_{E}&0&0&0&0\\0&0&S_{E}&0&0&0\end{pmatrix}
\nonumber\\&-r(S_{B},~S_{D})-r(S_{C},~S_{D})-r\begin{pmatrix}S_{E}\\S_{F}\end{pmatrix}-r\begin{pmatrix}S_{E}\\S_{G}\end{pmatrix}\nonumber\\=
&r_{b0ad00|ba00cd|0g0000|00f000|0e0000|00e000}\nonumber\\&-r_{bd}-r_{cd}-r_{e|f}-r_{e|g},
\end{align}
\begin{align}
r\begin{pmatrix} A_{49}-A_{69}&A_{4,10}\\
 A_{59}&A_{5,10}\end{pmatrix}
 =&r\begin{pmatrix}S_{B}&0&S_{A}&S_{D}&0\\S_{B}&S_{A}&0&0&S_{C}\\0&S_{G}&0&0&0\\0&0&S_{F}&0&0\\0&S_{E}&0&0&0\\0&0&S_{E}&0&0\\0&S_{F}&0&0&0\end{pmatrix}
-r(S_{C})-r(S_{D})-r\begin{pmatrix}S_{E}\\S_{F}\\S_{G}\end{pmatrix}-r\begin{pmatrix}S_{E}\\S_{F}\end{pmatrix} \nonumber\\&-m_{2}-m_{3}-m_{4}-m_{5} \nonumber\\
 =&r_{b0ad0|ba00c|0g000|00f00|0e000|00e00|0f000}\nonumber\\&-r_{c}-r_{d}-r_{e|f|g}-r_{e|f}-m_{2}-m_{3}-m_{4}-m_{5},
\end{align}
\begin{align}\label{equh373}
r\begin{pmatrix} A_{48}&A_{4,10}\\
 A_{58}&A_{5,10}\end{pmatrix}
 =&r\begin{pmatrix}S_{B}&0&S_{A}&S_{D}&0\\S_{B}&S_{A}&0&0&S_{C}\\0&S_{G}&0&0&0\\0&0&S_{F}&0&0\\0&0&S_{G}&0&0\\0&S_{E}&0&0&0\\0&0&S_{E}&0&0\end{pmatrix}
-r(S_{C})-r(S_{D})-r\begin{pmatrix}S_{E}\\S_{F}\\S_{G}\end{pmatrix}-r\begin{pmatrix}S_{E}\\S_{G}\end{pmatrix}\nonumber\\&-m_{3}-m_{4}-m_{5}\nonumber\\
 =&r_{b0ad0|ba00c|0g000|00f00|00g00|0e000|00e00}\nonumber\\&-r_{c}-r_{d}-r_{e|f|g}-r_{e|g}-m_{2}-m_{3}-m_{4}-m_{5}.
\end{align}
Hence from  (\ref{equh025})-(\ref{equh0210}) and (\ref{equh366})-(\ref{equh373}), we deduce that
\begin{align*}
t_{1}=p_{1}-r_{b}+r_{e}+m_{1}+m_{2}+m_{3}+m_{4}+r_{acd|e00}-r_{cd}-r_{e}=p_{1}+r_{acd|e00}-r_{bcd},
\end{align*}
\begin{align*}
t_{2}=p_{1}-r_{b}+r_{e}+n_{4}+n_{6}+n_{7}+n_{8}+r_{a|e|f|g}-r_{e|f|g}=p_{1}+r_{a|e|f|g}-r_{b},
\end{align*}
\begin{align*}
t_{3}=p_{1}+r_{b0ad0|ba00c|0g000|00f00|0e000|00e00}-r_{e|f|g}-r_{bc}-r_{bd},
\end{align*}
\begin{align*}
t_{4}=p_{1}-r_{b}+r_{e}+n_{4}+n_{6}+m_{1}+m_{2}+n_{8}+r_{ac|e0|g0}-r_{e|g}-r_{c}=p_{1}+r_{ac|e0|g0}-r_{bc},
\end{align*}
\begin{align*}
t_{5}=p_{1}-r_{b}+r_{e}+n_{4}+n_{6}+m_{1}+m_{3}+n_{7}+r_{ad|e0|f0}-r_{d}-r_{e|f}=p_{1}+r_{ad|e0|f0}-r_{bd},
\end{align*}
\begin{align*}
t_{7}=&r_{a|e|f|g}+r_{acd|e00}-r_{e}-r_{b0ad00|ba00cd|0g0000|00f000|0e0000|00e000}
\\&+r_{bd}+r_{cd}+r_{b0ad0|ba00c|0g000|00f00|0e000|00e00}+r_{e|f|g}
\\&+r_{ac|e0|g0}-r_{b0ad0|ba00c|0g000|00f00|00g00|0e000|00e00}-r_{ad0c|a0c0|e000|f000|g000},
\end{align*}
\begin{align*}
t_{8}=& r_{a|e|f|g}+r_{acd|e00}-r_{e}-r_{b0ad00|ba00cd|0g0000|00f000|0e0000|00e000}
\\&+r_{bd}+r_{cd}+r_{b0ad0|ba00c|0g000|00f00|0e000|00e00}+r_{e|f|g} \\&+
r_{ad|e0|f0}-r_{b0ad0|ba00c|0g000|00f00|0e000|00e00|0f000}-r_{ad00|a0cd|e000|f000|g000}.
\end{align*}
\end{proof}
Similarly, we can get the corresponding results on $W,Y,$ and $Z$.

\begin{theorem}Let $A\in \mathbb{H}^{m\times n},
B\in \mathbb{H}^{m\times p_{1}},C\in \mathbb{H}^{m\times p_{2}},D\in \mathbb{H}^{m\times p_{3}},E\in \mathbb{H}^{q_{1}\times n},F\in \mathbb{H}^{q_{2}\times n}$  and $G\in \mathbb{H}^{q_{3}\times n}$ be given. Assume that  equation (\ref{system001}) is consistent. Then,
\begin{align*}
\mathop {\max }\limits_{BX+WE+CYF+DZG=A } r \left({W }\right)
 =&\min\big\{  q_{1},~m,~q_{1}+r_{ab|f0|g0}-r_{e|f|g},~q_{1}+r_{abcd}-r_{e},\\&
 q_{1}+r_{ee0000|0ad0b0|a00c0b|g00000|0f0000}-r_{bcd}-r_{e|f}-r_{e|g},\\&\qquad \qquad q_{1}+r_{abd|f00}-r_{e|f},~q_{1}+r_{abc|g00}-r_{e|g}\big\},
\end{align*}
\begin{align*}
&\mathop {\min }\limits_{ BX+WE+CYF+DZG=A } r \left({W }\right)\nonumber
\\=&r_{abcd}+r_{ab|f0|g0}-r_{b}-r_{ee0000|0ad0b0|a00c0b|g00000|0f0000|0g0000}+r_{e|g}+r_{f|g}\\&+r_{ee0000|0ad0b0|a00c0b|g00000|0f0000}+r_{bcd}\\&+
\min\big\{r_{abd|f00}-r_{ee00000|0ad00b0|a00cd0b|g000000|0f00000}-r_{aadcb|g0000|0f000|f0000},\\&\qquad \qquad
r_{abc|g00}-r_{ee00000|0ad0b0c|a00c0b0|g000000|0f00000}-r_{aabcd|g0000|0f000|0g000} \big\}.
\end{align*}
\end{theorem}

\begin{theorem}Let $A\in \mathbb{H}^{m\times n},
B\in \mathbb{H}^{m\times p_{1}},C\in \mathbb{H}^{m\times p_{2}},D\in \mathbb{H}^{m\times p_{3}},E\in \mathbb{H}^{q_{1}\times n},F\in \mathbb{H}^{q_{2}\times n}$  and $G\in \mathbb{H}^{q_{3}\times n}$ be given. Assume that  equation (\ref{system001}) is consistent.
Then,
\begin{align*}
&\mathop {\max }\limits_{BX+WE+CYF+DZG=A } r \left({Y }\right)
=\\&\min\big\{  p_{2},~q_{2},~p_{2}+q_{2}+r_{ab|e0|g0}-r_{e|f|g}-r_{bc},~p_{2}+q_{2}+r_{abd|e00}-r_{e|f}-r_{bcd}\big\},\\
&
\mathop {\min }\limits_{ BX+WE+CYF+DZG=A } r \left({Y }\right)
=r_{ab|e0|g0}+r_{abd|e00}-r_{abd|e00|g00}-r_{b}-r_{e},
\\&\mathop {\max }\limits_{BX+WE+CYF+DZG=A } r \left({Z }\right)
=\\&\min\big\{  p_{3},~q_{3},~p_{3}+q_{3}+r_{ab|e0|f0}-r_{e|f|g}-r_{bd},~p_{3}+q_{3}+r_{abc|e00}-r_{e|g}-r_{bcd}\big\},
\\&
\mathop {\min }\limits_{ BX+WE+CYF+DZG=A } r \left({Z }\right)
=r_{ab|e0|f0}+r_{abc|e00}-r_{abc|e00|f00}-r_{b}-r_{e}.
\end{align*}
\end{theorem}

\begin{remark}
All the results are true over octonion algebra. 
\end{remark}

\section{\textbf{Conclusion}}

We have established the  simultaneous decomposition of the general real quaternion  matrix array (\ref{array1}). We have derived all the
dimensions of identity matrices in the equivalence canonical form of the   matrix array (\ref{array1}). Using the
simultaneous decomposition of the general  matrix array (\ref{array1}), we have presented necessary and sufficient conditions for the existence and
the general solutions to the real matrix equations (\ref{system002}) and (\ref{system001}), respectively.  Moreover, we have given the
range of ranks  of the   general solutions to (\ref{system002}) and (\ref{system001}), respectively.

As a special case of the   matrix array (\ref{array1}), we have derived all the
dimensions of identity matrices in the   equivalence canonical form of triple matrices with the same row or
column numbers,  which perfect the results  in \cite{QWWangandyushaowen}. The presented expression of the general solution  is more useful   than the expression found in \cite{QWWangandyushaowen}, since the latter can not be used to consider the maximal and minimal ranks of the general solution to (\ref{system002}). On the other hand, Wang and He \cite{wanghe} gave
the range of ranks of $Y$ and $Z$ to (\ref{system001}), but the two present authors did not derive the range of ranks of $X$ and $W$. We in this paper have solved this problem.

\end{document}